\documentclass[twoside,10pt]{article}

\usepackage[T1]{fontenc}
\usepackage[utf8]{inputenc}
\usepackage[helvratio=0.88, tighter]{newtxtext}
\usepackage{microtype}

\usepackage{geometry}
\usepackage{amsmath,amsthm}
\usepackage{amssymb}
\usepackage{color,tikz,siunitx,enumerate,xfrac,dsfont,bropd,etoolbox,xparse,pgffor}
\usetikzlibrary{arrows.meta,calc,decorations.pathreplacing,positioning}

\newtheoremstyle{plain}{6.5pt}{6.5pt}{\itshape}{}{\bfseries}{.}{.5em}{}
\newtheoremstyle{definition}{6.5pt}{6.5pt}{\normalfont}{}{\bfseries}{.}{.5em}{}

\theoremstyle{plain}
\newtheorem{theorem}{Theorem}[section]
\newtheorem{lemma}[theorem]{Lemma}
\newtheorem{corollary}[theorem]{Corollary}
\newtheorem{proposition}[theorem]{Proposition}

\theoremstyle{definition}
\newtheorem{definition}[theorem]{Definition}
\newtheorem{assumption}[theorem]{Assumption}
\newtheorem{remark}[theorem]{Remark}
\newtheorem{example}[theorem]{Example}

\newtheorem{innercustomgeneric}{\customgenericname}
\providecommand{\customgenericname}{}
\newcommand{\newcustomtheorem}[2]{%
    \newenvironment{#1}[1]
    {%
        \renewcommand\customgenericname{#2}%
        \renewcommand\theinnercustomgeneric{##1}%
        \innercustomgeneric
    }
    {\endinnercustomgeneric}
}
\newcustomtheorem{customthm}{Theorem}

\numberwithin{equation}{section}

\makeatletter
\newbox\emsqedbox
\AtBeginDocument{%
    \setbox\@tempboxa\hbox{x}%
    \setbox\emsqedbox\hbox{%
        \vrule\@width\ht\@tempboxa
              \@height\ht\@tempboxa
              \@depth\z@
    }%
}
\renewcommand*\qedsymbol{%
    \relax
    \ifmmode
        \copy\emsqedbox
    \else
        \unhcopy\emsqedbox
    \fi
}
\makeatother

\usepackage[shortlabels]{enumitem}
\setlist{
    font=\upshape,
    itemsep=0.25\baselineskip,
    leftmargin=\parindent,
    parsep=0pt,
    topsep=0.25\baselineskip
}

\makeatletter
\renewcommand*\@seccntformat[1]{\csname the#1\endcsname.\enskip}

\renewcommand*\section{%
    \@startsection{section}{1}{\z@}%
        {-13pt \@plus0pt \@minus-3.25pt}%
        {13pt \@plus0pt \@minus3.25pt}%
        {\raggedright\normalfont\large\bfseries\boldmath}%
}

\renewcommand*\subsection{%
    \@startsection{subsection}{2}{\z@}%
        {-13pt \@plus0pt \@minus-3.25pt}%
        {6.5pt}%
        {\raggedright\normalfont\bfseries\boldmath}%
}
\makeatother

\usepackage[pdfa, colorlinks, allcolors=blue]{hyperref}

\makeatletter

\let\keywords\relax
\let\msc\relax
\let\abstract\relax

\gdef\@keywords{}
\gdef\@msc{}

\newcommand{\keywords}[1]{%
    \gdef\@keywords{\textit{Keywords:} #1.}%
}

\newcommand{\msc}[2][]{%
    \gdef\@msc{%
        \textit{Mathematics Subject Classification 2020:} #2%
        \ifx&#1&\else\ (#1)\fi.%
    }%
}

\newcommand{\ems@titlefootnote}[1]{%
    \insert\footins{%
        \reset@font\footnotesize
        \interlinepenalty\interfootnotelinepenalty
        \splittopskip\footnotesep
        \splitmaxdepth\dp\strutbox
        \floatingpenalty\@MM
        \hsize\columnwidth
        \@parboxrestore
        \color@begingroup
            \noindent\strut #1\par
        \color@endgroup
    }%
}

\newbox\emsabstractbox
\newenvironment{abstract}{%
    \global\setbox\emsabstractbox\vbox\bgroup
        \normalcolor\small
        \noindent\strut\textbf{Abstract.}\enskip\ignorespaces
}{%
        \unskip\strut\par
    \egroup
}

\renewcommand{\maketitle}{%
    \thispagestyle{empty}
    \begingroup
        \vspace*{26pt}
        \centering
        {\Large\bfseries\boldmath \@title \par}
        \vspace{8mm}
        {\large \@author \par}
        \vspace{5mm}
    \endgroup
    \ifx\@msc\@empty
    \else
        \ems@titlefootnote{\@msc}%
    \fi
    \ifx\@keywords\@empty
    \else
        \ems@titlefootnote{\@keywords}%
    \fi
    \ifvoid\emsabstractbox
    \else
        \noindent\unvbox\emsabstractbox\par
    \fi
    \vspace{26pt}
}

\newcounter{ems@auth}
\gdef\@author{}

\newcommand{\emsauthor}[3]{%
    \stepcounter{ems@auth}%
    \begingroup
        \def\givenname##1{\expandafter\gdef\csname auth@#1@first\endcsname{##1}}%
        \def\surname##1{\expandafter\gdef\csname auth@#1@last\endcsname{##1}}%
        \def\mrid##1{\expandafter\gdef\csname auth@#1@mrid\endcsname{##1}}%
        \def\zblid##1{\expandafter\gdef\csname auth@#1@zblid\endcsname{##1}}%
        \def\orcid##1{\expandafter\gdef\csname auth@#1@orcid\endcsname{##1}}%
        #2%
    \endgroup
    \ifnum\value{ems@auth}=1
        \expandafter\gdef\expandafter\@author\expandafter{%
            \csname auth@#1@first\endcsname\ %
            \csname auth@#1@last\endcsname
        }%
    \else
        \expandafter\g@addto@macro\expandafter\@author\expandafter{%
            , \csname auth@#1@first\endcsname\ %
            \csname auth@#1@last\endcsname
        }%
    \fi
}

\def\ems@setmeta#1#2#3{%
    \expandafter\gdef\csname affil@#1@#2\endcsname{#3}%
}

\def\ems@makeaffilmacro#1{%
    \expandafter\def\csname #1\endcsname##1{%
        \@ifnextchar\bgroup
            {\csname ems@#1@two\endcsname{##1}}%
            {\csname ems@#1@one\endcsname{##1}}%
    }%
    \expandafter\def\csname ems@#1@two\endcsname##1##2{%
        \ems@setmeta{\current@affil@id @##1}{#1}{##2}%
    }%
    \expandafter\def\csname ems@#1@one\endcsname##1{%
        \ems@setmeta{\current@affil@id @1}{#1}{##1}%
    }%
}

\newcommand{\Emsaffil}[2]{%
    \def\current@affil@id{#1}%
    \begingroup
        \ems@makeaffilmacro{department}%
        \ems@makeaffilmacro{organisation}%
        \ems@makeaffilmacro{rorid}%
        \ems@makeaffilmacro{address}%
        \ems@makeaffilmacro{zip}%
        \ems@makeaffilmacro{city}%
        \ems@makeaffilmacro{country}%
        \ems@makeaffilmacro{affemail}%
        #2%
    \endgroup
}

\def\ems@printid#1#2#3{%
    \ifcsdef{auth@\aid @#1}{%
        \edef\tempval{\csname auth@\aid @#1\endcsname}%
        \expandafter\ifblank\expandafter{\tempval}{}{%
            #2 \href{#3\tempval}{\tempval}\space
            \gdef\hasid{1}%
        }%
    }{}%
}

\AtEndDocument{%
    \par \addvspace{19.5pt}
    \begingroup
    \small \raggedright
    \ifnum\value{ems@auth}>0
        \foreach \i in {1,...,\value{ems@auth}} {%
            \def\aid{\i}%
            {\bfseries
                \csname auth@\aid @first\endcsname\ %
                \csname auth@\aid @last\endcsname
            }\par\nobreak
            \foreach \j in {1,2,3,4} {
                \def\affkey{\aid @\j}%
                \@ifundefined{affil@\affkey @organisation}{}{%
                    \@ifundefined{affil@\affkey @department}{}{%
                        \csname affil@\affkey @department\endcsname, %
                    }%
                    \csname affil@\affkey @organisation\endcsname\par
                    \@ifundefined{affil@\affkey @address}{}{%
                        \csname affil@\affkey @address\endcsname, %
                    }%
                    \@ifundefined{affil@\affkey @zip}{}{%
                        \csname affil@\affkey @zip\endcsname\ %
                    }%
                    \@ifundefined{affil@\affkey @city}{}{%
                        \csname affil@\affkey @city\endcsname, %
                    }%
                    \@ifundefined{affil@\affkey @country}{}{%
                        \csname affil@\affkey @country\endcsname
                    }%
                    \par
                    \ifcsdef{affil@\affkey @affemail}{%
                        \edef\tempemail{\csname affil@\affkey @affemail\endcsname}%
                        \expandafter\ifblank\expandafter{\tempemail}{}{%
                            \href{mailto:\tempemail}{\tempemail}\par
                        }%
                    }{}%
                }%
            }%
            \gdef\hasid{0}%
            \setbox0=\hbox{%
                \ems@printid{zblid}{zbMATH}{https://zbmath.org/authors/}%
                \ems@printid{mrid}{MR}{https://mathscinet.ams.org/mathscinet/MRAuthorID/}%
                \ems@printid{orcid}{ORCID}{https://orcid.org/}%
            }%
            \ifnum\hasid=1
                \noindent Author IDs:\space \unhbox0 \par
            \fi
            \vspace{13pt}%
        }%
    \fi
    \endgroup
}

\makeatother

\makeatletter
\renewcommand*\thebibliography[1]{%
    \section*{References}%
    \list{\@biblabel{\@arabic\c@enumiv}}{%
        \small
        \settowidth\labelwidth{\@biblabel{#1}}%
        \leftmargin \dimexpr\labelwidth+\labelsep\relax
        \itemsep 0pt
        \parsep 0pt
        \usecounter{enumiv}%
        \let\p@enumiv\@empty
        \renewcommand\theenumiv{\@arabic\c@enumiv}%
    }%
    \interlinepenalty \@M
    \emergencystretch 1em
    \sfcode`\.\@m
}
\makeatother

\newcommand*\MR[1]{%
    , \href{https://mathscinet.ams.org/mathscinet-getitem?mr=#1}{MR #1}%
}
\newcommand*\Zbl[1]{%
    , \href{https://zbmath.org/?q=an:#1}{Zbl #1}%
}

\newcommand{\dx}[1]{\mathop{}\!\mathrm{d}#1}
\newcommand{\dH}[1]{\mathop{}\!\mathrm{d}\mathcal{H}^{#1}}

\DeclareMathOperator{\dist}{dist}

\def\XXint#1#2#3{%
    {%
        \setbox0=\hbox{$#1{#2#3}{\int}$}%
        \vcenter{\hbox{$#2#3$}}%
        \kern-.53\wd0
    }%
}

\newcommand{\R}{\mathbb{R}}

\newcommand{\cH}{\ensuremath{\mathcal H}}

\newcommand{\cJ}{\ensuremath{\mathcal J}}

\let\ge\geqslant
\let\le\leqslant
\let\emptyset\varnothing
\newcommand{\mres}{\mathbin{\vrule height 1.6ex depth 0pt width 0.13ex\vrule height 0.13ex depth 0pt width 1.3ex}}

\begin{document}

\title{
    Full-Density Degenerate Stagnation Points for Water Waves with General Vorticity
}


\emsauthor{1}{
    \givenname{Lili}
    \surname{Du}
}{L.~Du}

\emsauthor{2}{
    \givenname{Chunlei}
    \surname{Yang}
}{C.~Yang}

\Emsaffil{1}{
    \department{Department of Mathematics}
    \organisation{Sichuan University}
    \rorid{01a2bcd34}
    \address{No. 24, Wuhou District}
    \zip{}
    \city{Chengdu}
    \country{China}
    \affemail{dulili@scu.edu.cn}
}

\Emsaffil{2}{
    \department{School of Mathematical Sciences}
    \organisation{Shenzhen University}
    \rorid{}
    \address{No. 3688, Nanhai Avenue}
    \zip{10001}
    \city{Shenzhen}
    \country{China}
    \affemail{yangchunlei@szu.edu.cn}
}

\msc[35R35,76D07]{35Q35}

\begin{abstract}
    In this paper, we revisit the singular asymptotics of the free surface near stagnation points for two-dimensional traveling gravity water waves with vorticity. We prove the nonexistence of full-density degenerate stagnation points beyond the strict two-sided linear growth regime. 
    
    Our main tools are a modified Weiss-type monotonicity formula and a modified Almgren-type frequency formula. Together, they provide a new approach that completely avoids the use of a Bessel-type differential inequality, which is an essential tool used in the previous literature to prove the nonexistence of full-density degenerate stagnation points (\emph{Ann. I. H. Poincaré-AN}, \textbf{29}, 861--885, 2012).
    
    As consequences, we obtain uniform bounds for frequency-normalized blow-ups in arbitrary dimension and strong convergence in dimension two. As an application, we extend the Stokes conjecture for rotational waves to a broader class of vorticity distributions.
\end{abstract}

\maketitle

\tableofcontents

\medskip

\section{Introduction and main results}

In this paper, we study solutions to the following semilinear Bernoulli-type free boundary problem
\begin{equation}\label{eq.wwp}
    \left\{
        \begin{aligned}
            \Delta u &= -f(u) && \text{ in } \Omega \cap \{u>0\}, \\
            |\nabla u|^2 &= x_n && \text{ on } \Omega \cap \partial\{u>0\}.
        \end{aligned}
    \right.
\end{equation} 
Here \(\Omega\) is an open and bounded domain in \(\R^n\) (\(n\ge 2\)) which has a nonempty intersection with the hyperplane \(\{x_n=0\}\). The free boundary \(\partial\{u>0\}\cap\Omega\) is understood to exist in the upper half-plane \(\{x_n\ge 0\}\). 

In two dimensions, the unknown function \(u\) in \eqref{eq.wwp} can be viewed as the two-dimensional stream function and the problem \eqref{eq.wwp} models the motion of an incompressible inviscid rotational fluid with a free surface under the influence of gravity. The connection between problem \eqref{eq.wwp} and the corresponding Euler equations for incompressible inviscid fluids can be found in \cite{CS04}.

We are particularly interested in solutions with singularities arsing at stagnation points, at which the gradient \(|\nabla u|\) vanishes. It follows from the second equation in \eqref{eq.wwp} that such points lie on the intersection \(\partial\{u>0\} \cap \{x_n=0\}\). In order to study the behavior of the free surface near such points, V\u{a}rv\u{a}ruc\u{a} and Weiss \cite{VW12} studied the blow-up sequence
\begin{equation}\label{eq.rescaling}
    u_r(x) := \frac{u(x^0 + r x)}{r^{3/2}}, \qquad x^0 \in \partial\{u>0\} \cap \{x_n=0\}.
\end{equation}
Let \(f(z)\) be continuous and define its primitive by \(F(z):=\int_0^zf(s)\,\dx{s}\). V\u{a}rv\u{a}ruc\u{a} and Weiss derived the following monotonicity formula.

\begin{theorem}[cf. Theorem 3.4 in \cite{VW12}]\label{thm:monotonicity_formula}
    Let \(u\) be a variational solution (cf. Definition \ref{def:variational_solution}) of \eqref{eq.wwp}, let \(x^0\in \Omega\) be such that \(x_n^0=0\), and let \(\delta:=\dist(x^0,\partial\Omega)/2\). For every \(r\in (0,\delta)\), define
    \[
        I_{x^0,u}(r)=I(r)=\int_{B_r(x^0)}\left(|\nabla u|^2-uf(u)+x_n\,\chi_{\{u>0\}}\right)\,\dx{x}.
    \]
    and 
    \begin{equation}\label{eq.Jr_weiss}
        J_{x^0,u}(r)=J(r)=\int_{\partial B_r(x^0)} u^2\,\dH{n-1}.
    \end{equation}
    Set 
    \begin{equation}\label{eq.Phir}
        M_{x^0,u}(r)=M(r)=r^{-n-1}I(r)-\frac 32r^{-n-2}J(r).
    \end{equation}
    Then for a.e. \(r\in (0,\delta)\), we have
    \begin{equation}\label{mono.Wei.model1}
        M'(r) = 2r^{-n-1}\int_{\partial B_r(x^0)}\left( \nabla u\cdot\nu - \frac{3}{2}\frac{u}{r} \right)^2\,\dH{n-1} + r^{-n-2}K(r),
    \end{equation}
    where 
    \begin{equation}\label{eq.Kr}
    \begin{aligned}
            K_{x^0,u}(r) = K(r) &= r\int_{\partial B_r(x^0)}\left[ 2F(u)-uf(u) \right]\,\dH{n-1} \\ 
            &\quad+ \int_{B_r(x^0)}\left[ (n-2)uf(u)-2nF(u) \right]\,\dx{x}.
        \end{aligned}
    \end{equation}
\end{theorem}

It should be noted that the remainder \(K(r)\) has no prescribed sign, and hence \(r\mapsto M(r)\) need not be monotone. Nevertheless, under the growth assumption
\begin{equation}\label{eq.growth_assumption}
    |\nabla u|^2\le Cx_n^+\quad\text{ in }\Omega\cap\{u>0\},
\end{equation}
the continuity of \(f\) implies that \(r\mapsto r^{-n-2}K(r)\) is locally an integrable function at any stagnation point. Consequently, it follows from the almost-monotonicity identity \eqref{mono.Wei.model1} that the limit \(M(0^+):=\lim_{r\to 0^+}M(r)\) exists and is finite. Then recall the rescaling defined in \eqref{eq.rescaling}. For every sequence \(r_j\to 0^+\), one may extract a non-relabeled subsequence such that \(u_{r_j}\) converges locally uniformly and strongly in \(W_{\rm loc}^{1,2}(\R^n)\) to a nonnegative \(3/2\)-homogeneous blow-up limit \(u_0\). In two dimensions (\(n=2\)), the explicit form of \(u_0\) can be fully derived, and the stagnation points are classified as either \emph{non-degenerate} or \emph{degenerate}. More precisely, if \(x^0\) is a non-degenerate point, then \(u_0\not\equiv 0\) and 
\[
    \{u_0>0\}=\bigl\{
        (\rho,\theta):\rho>0,\quad \tfrac{\pi}6<\theta<\tfrac{5\pi}{6}
    \bigr\}.
\]
Thus, the positivity set of \(u_0\) is a symmetric cone with opening angle \(2\pi/3\), centered on the positive \(x_2\)-axis. This homogeneous profile is known as the \emph{Stokes corner} when the free surface is assumed to be a continuously injective curve (see Figure \ref{fig.stokes_corner}).

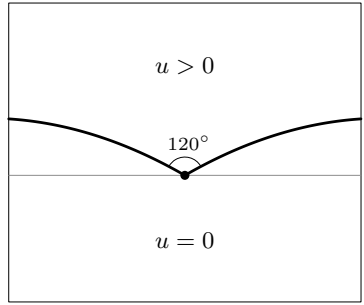
\begin{figure}[!htbp]
    \centering
    \begin{tikzpicture}[
        x=.93cm,
        y=.93cm,
        line cap=round,
        line join=round,
        every node/.style={font=\small},
        panel/.style={draw=black,line width=.35pt},
        level/.style={draw=black!45,line width=.35pt},
        freebdry/.style={draw=black,line width=1.05pt}
    ]
        \def\xL{-2.15}
        \def\xR{ 2.15}
        \def\yB{-1.20}
        \def\yT{ 2.05}

        \begin{scope}[shift={(-2.65,2.80)}]
            \draw[panel]
                (-2.5,-1.8) rectangle (2.5,2.45);

            \draw[level] (-2.5,0) -- (2.5,0);

            \draw[freebdry]
                plot[domain=-2.5:0,samples=80]
                (\x,{-\x/sqrt(3)-.1025*\x*\x});

            \draw[freebdry]
                plot[domain=0:2.5,samples=80]
                (\x,{\x/sqrt(3)-.1025*\x*\x});

            \draw[line width=.35pt]
                (30:.26)
                arc[start angle=30,end angle=150,radius=.26];

            \node[
                font=\scriptsize,
                fill=white,
                inner sep=.4pt
            ] at (.06,.47) {$120^\circ$};

            \node at (0,1.55) {$u>0$};
            \node at (0,-.92) {$u=0$};

            \fill (0,0) circle (1.7pt);
        \end{scope}
    \end{tikzpicture}
    \caption{Stokes corner}
    \label{fig.stokes_corner}
\end{figure}

Conversely, if \(x^0\) is a degenerate stagnation point, then every \(3/2\)-homogeneous blow-up satisfies
\(u_0\equiv0\). In this case, the limiting configuration is determined by the phase indicator. More precisely, after passing to a subsequence,
\[
    \chi_{\{u_{r_j}>0\}}
    \longrightarrow
    \chi_0
    \qquad\text{in }L^1_{\mathrm{loc}}(\mathbb R^2),
\]
where \(\chi_0\in\{0,1\}\) a.e. The weighted density is then
\[
    M(0^+)=\int_{B_1}x_2^+\chi_0\,\dx{x}.
\]
Due to the natural of \(\chi_0\), the quantity \(M(0^+)\) can be interpreted as the weighted density of the positivity set at the degenerate stagnation point \(x^0\), with weight \(x_2^+\). In particular, if \(\chi_0\equiv 0\), then \(M(0^+)=0\), and such a degenerate stagnation point is called a \emph{cusp point}. Under the assumption that the free surface is an injective continuous curve, the asymptotic behavior of the free surface near a cusp point is illustrated in Figure \ref{fig.right_left_cusp}. Singularities of this type only exist for rotational waves. For the irrotational case \(f\equiv 0\), it is proved in \cite{VW11} that if \(C=1\) in \eqref{eq.growth_assumption}, then cusp points do not exist. We also refer readers to the recent work \cite{McC25} that removes the restriction \(C=1\). For cusp singularities of rotational waves, V\u{a}rv\u{a}ruc\u{a} and Weiss \cite{VW12} further conjectured that cusp configurations can be ruled out under the Rayleigh-Taylor sign condition 
\[
    |\nabla u|^2+2F(u)-x_n^+\le 0\quad\text{ in }\Omega\cap\{u>0\}.
\] 

\begin{figure}[!htbp]
    \centering
    \begin{tikzpicture}[
        x=.93cm,
        y=.93cm,
        line cap=round,
        line join=round,
        every node/.style={font=\small},
        panel/.style={draw=black,line width=.35pt},
        level/.style={draw=black!45,line width=.35pt},
        freebdry/.style={draw=black,line width=1.05pt}
    ]
        \begin{scope}[shift={(-3.25,0)}]
            \draw[panel]
                (-2.5,-1.8) rectangle (2.5,2.45);

            \draw[level] (-2.5,0) -- (2.5,0);

            \draw[freebdry]
                plot[domain=-2.5:0,samples=80]
                (\x,{.22*\x*\x});

            \draw[freebdry]
                plot[domain=-2.5:0,samples=80]
                (\x,{.091*\x*\x});

            \node at (.05,1.58) {$u=0$};
            \node at (.05,-.92) {$u=0$};

            \node[
                font=\tiny,
                fill=white,
                inner sep=.02pt
            ] at (-2.16,.65) {$u>0$};

            \fill (0,0) circle (1.7pt);

            \node[font=\footnotesize] at (0,-2.15)
                {(a) Left cusp};
        \end{scope}

        \begin{scope}[shift={(3.25,0)}]
            \draw[panel]
                (-2.5,-1.8) rectangle (2.5,2.45);

            \draw[level] (-2.5,0) -- (2.5,0);

            \draw[freebdry]
                plot[domain=0:2.5,samples=80]
                (\x,{.22*\x*\x});

            \draw[freebdry]
                plot[domain=0:2.5,samples=80]
                (\x,{.091*\x*\x});

            \node at (-.05,1.58) {$u=0$};
            \node at (-.05,-.92) {$u=0$};

            \node[
                font=\tiny,
                fill=white,
                inner sep=.02pt
            ] at (2.17,.66) {$u>0$};

            \fill (0,0) circle (1.7pt);

            \node[font=\footnotesize] at (0,-2.15)
                {(b) Right cusp};
        \end{scope}
    \end{tikzpicture}
    \caption{Cusp asymptotics}
    \label{fig.right_left_cusp}
\end{figure}
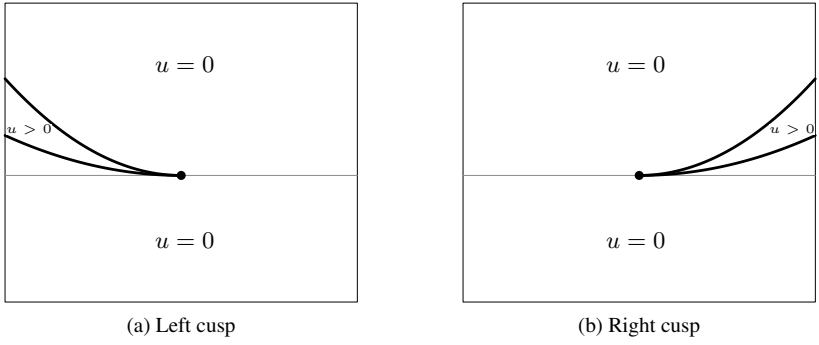

If \(\chi_0=1\), a direct computation yields \(M(0^+)=\int_{B_1}x_2^+\,\dx{x}=2/3\), and we call such a degenerate stagnation point a \emph{horizontally flat point}.\footnote{In general dimensions, such singularities are usually called full-density singularities or highest-density singularities. See \cite{KW25,WZ10} and references therein.} In this case, the free boundary is asymptotically flat at the stagnation point (see Figure \ref{fig.horizontal_flat}). Horizontally flat points can arise when the vorticity is negative at the free surface (see, e.g., \cite[Remark 6.4 (i)]{VW12}). V\u{a}rv\u{a}ruc\u{a} and Weiss emphasized that ``\emph{of particular difficulty is the case when the vorticity is \(0\) at the free surface, and may have infinitely many sign changes accumulating there}'' \cite[p.~863]{VW12}. In order to exclude horizontally flat points, V\u{a}rv\u{a}ruc\u{a} and Weiss \cite[Sec. 6-10]{VW12} first work in the general \(n\)-dimensional setting, and define the set of horizontally flat points by
\[
    \Sigma^u:=\Bigl\{
        x^0\in \partial\{u>0\}\cap\{x_n=0\} : M(0^+) = \int_{B_1} x_n^+ \, \dx{x}
    \Bigr\}.
\]

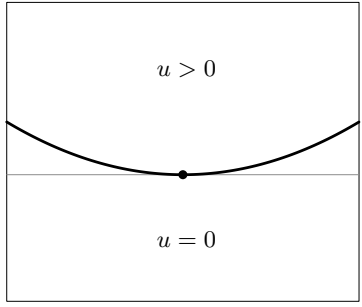
\begin{figure}[!htbp]
    \centering
    \begin{tikzpicture}[
        x=.93cm,
        y=.93cm,
        line cap=round,
        line join=round,
        every node/.style={font=\small},
        panel/.style={draw=black,line width=.35pt},
        level/.style={draw=black!45,line width=.35pt},
        freebdry/.style={draw=black,line width=1.05pt}
    ]
    \begin{scope}[shift={(2.65,-2.80)}]
            \draw[panel]
                (-2.5,-1.8) rectangle (2.5,2.45);

            \draw[level] (-2.5,0) -- (2.5,0);

            \draw[freebdry]
                plot[domain=-2.5:2.5,samples=80]
                (\x,{.120*\x*\x});

            \node at (.05,1.50) {$u>0$};
            \node at (.05,-.92) {$u=0$};

            \fill (0,0) circle (1.7pt);
        \end{scope}
    \end{tikzpicture}
    \caption{Horizontally flat point}
    \label{fig.horizontal_flat}
\end{figure}

For every \(x^0\in\Sigma^u\), they introduce an Almgren-type frequency formula
\begin{equation}\label{eq.freqDr}
    D_{x^0,u}(r) = D(r) = \frac{r \int_{\partial B_r(x^0)} u \, \nabla u\cdot\nu \, \dH{n-1}}
    {\int_{\partial B_r(x^0)} u^2 \, \dH{n-1}},
\end{equation}
together with the associated normalized rescaling
\begin{equation}\label{eq.phir}
    \phi_r(x) = \frac{u(x^0 + r x)}{A(r)}, \quad 
    A(r) = \left( r^{1-n}\int_{\partial B_r(x^0)} u^2 \, \dH{n-1} \right)^{1/2}.
\end{equation}
The following theorem summarizes the main results for points in \(\Sigma^u\), proved in \cite[Proposition 7.1 and Theorem 9.1]{VW12} under the growth assumption \eqref{eq.growth_assumption} and the condition \(|f(z)|\le Cz\) for all \(z\in(0,z_0)\).

\begin{theorem}[cf. Proposition 7.1 and Theorem 9.1 in \cite{VW12}]\label{thm.VWflat}
    Let \(u\) be a variational solution to the problem \eqref{eq.wwp}, and let \(x^0\in\Sigma^u\). Assume that there exists a constant \(C > 0\) such that 
    \begin{equation}\label{eq.fgrowth}
        |f(z)| \le Cz\qquad\text{ for all }z\in (0,z_0).
    \end{equation}
    Then the limit \(D(0^+) := \lim_{r\to 0^+} D(r)\) exists and satisfies \(D(0^+) \ge 3/2\). Moreover, the blow-up sequence \(\{\phi_r\}_{r>0}\) defined in \eqref{eq.phir} is uniformly bounded in \(W^{1,2}(B_1)\) with 
    \[
        \int_{B_1}|\nabla\phi_r|^2\,\dx{x} + \int_{B_1}\phi_r^2\,\dx{x} \le C_nD(0^+),
    \]
    for all \(r>0\) sufficiently small, where \(C_n>0\) depends only on \(n\).

    Furthermore if \(n=2\), the convergence \(\phi_r \to \phi_0\) is strong in \(W^{1,2}_{\mathrm{loc}}(B_1\setminus\{0\})\) and the blow-up limit \(\phi_0\) is of the form
    \[
        \phi_0(x) =  
        \sqrt{\frac{2}{\pi}}\,\rho^{N(x^0)}
        \left|\sin\left(N(x^0)\min\{\max\{\theta,0\},\pi\}\right)\right|,
    \]
    where \(N(x^0)\ge 2\) is an \emph{integer} that depends on \(x^0\).
\end{theorem}

As a direct application of the above result, V\u{a}rv\u{a}ruc\u{a} and Weiss proved the nonexistence of horizontal flat singularity when the free surface is an injective curve.

\begin{corollary}
    Let \(n=2\) and let \(u\) be a variational solution to the problem \eqref{eq.wwp}. Assume that the free surface \(\partial\{u>0\}\) is an injective curve in a neighborhood of \(x^0\in\Sigma^u\). Assume also that either \(f(z)\ge 0\) for all \(z\) in a right neighborhood of \(0\) or \(|f(z)| \le C z\) for \(z \in [0,z_0]\). Then \(\Sigma^u=\emptyset\).
\end{corollary}

We note that if \(f(z)\ge0\) in a right neighborhood of \(0\), then the exclusion of horizontally flat points does not rely on Theorem \ref{thm.VWflat}, but rather on the boundary point principle of Oddson established in \cite{Odd68}; see also Remark 6.3 and Proposition 6.5 in \cite{VW12}. 

We now explain the role of assumption \eqref{eq.fgrowth} throughout the frequency analysis. Under \eqref{eq.fgrowth}, we have \(|F(z)|\le Cz^2\), which gives the key estimate
\begin{equation}\label{eq.esmVW}
    |K(r)|\le Cr\int_{\partial B_r(x^0)}u^2\,\dH{n-1},
\end{equation}
for every \(x^0\in\Sigma^u\), with \(K(r)\) as defined in \eqref{eq.Kr}. We next introduce an auxiliary function
\[
    Y(r):=\int_0^r t^{-n-1}\int_{\partial B_t(x^0)} u^2\,\dH{n-1}\dx{t}.
\]
Using \eqref{eq.esmVW} and the monotonicity formula in Theorem \ref{thm:monotonicity_formula}, one shows that \(Y(r)\) satisfies a Bessel-type differential inequality
\[
    \frac{\dx{}}{\dx{r}}\left(\frac{Y'(r)}r\right)
    \ge
    -\alpha\frac{Y(r)}r,\qquad 0<\alpha<\infty,
\]
where the constant \(\alpha<\infty\) depends on \(C\) from \eqref{eq.fgrowth}. Crucially, this implies that for all sufficiently small \(r>0\), 
\[
    \frac{\dx{}^2}{\dx{r^2}}\left(\frac{Y(r)}{\sqrt{r}}\right)
    \ge 
    \frac{\tfrac34-\alpha r^2}{r^{5/2}}Y(r)\ge 0,
\]
which justifies the convexity argument underlying the lower bound estimates for the frequency function \(D(r)\); see \cite[Theorem 6.12]{VW12}.

The main objective of this paper is to relax the assumption \eqref{eq.fgrowth} on \(f\) and push the theoretical boundaries of the frequency analysis for full-density degenerate stagnation points developed in \cite{VW12}, while bypassing the need for Bessel-type differential inequalities. In the process, we introduce an interesting modified Weiss-type monotonicity formula and a modified Almgren-type frequency formula, which are of independent interest.

\subsection{Main results}
We now state our main results. Let \(z_0>0\) and let \(f\in C^0((0,z_0))\).

\begin{assumption}\label{assumption:nonlinearity}
    We denote \(f^\pm(z) := \max\{\pm f(z),0\}\) and assume that there exist nonnegative finite constants \(\mathsf P_0\) and \(\mathsf N_0\) such that
    \begin{equation}\label{eq.P0N0}
        f^+(z)\le\mathsf P_0\,z,\qquad 
        f^-(z)\le\mathsf N_0\,z\bigl(1+|\log z|\bigr),\qquad\text{ for all }z\in(0,z_0).
    \end{equation}
    Set \(F(z):=\int_0^zf(s)\,\dx{s}\), and we assume in addition that there exists a nonnegative finite constant \(\mathsf W_0\) such that
    \begin{equation}\label{eq.W0}
        |zf(z)-2F(z)|\le\mathsf W_0\,z^2\qquad\text{ for all }z\in(0,z_0).
    \end{equation}
\end{assumption}

Before displaying our main results, we comment on the above assumption. 

\begin{remark}
    The one-sided bounds on \(f^\pm(z)\) in \eqref{eq.P0N0} give 
    \[
        |f(z)|\le (\mathsf P_0+\mathsf N_0)\, z
        +
        \mathsf N_0\, z|\log z|,\qquad\text{ for all }z\in(0,z_0),
    \]
    so \(f\in L^1((0,z_0))\) and \(\lim_{z\to 0^+}f(z)=0\). Thus \(f\) extends continuously to \(0\) by setting \(f(0):=0\). In particular, the primitive \(F(z)\) is well-defined with \(F(0)=0\). 
\end{remark}

\begin{remark}
    If \(f(z)\) satisfies \eqref{eq.fgrowth}, then it is easy to check that Assumption \ref{assumption:nonlinearity} holds with \(\mathsf P_0=\mathsf N_0=C\) and \(\mathsf W_0=2C\). In particular, the assumption \eqref{eq.fgrowth} is a special case of Assumption \ref{assumption:nonlinearity}. It should also be noted that the conditions \eqref{eq.P0N0} do not imply \eqref{eq.W0}, but \eqref{eq.W0} is automatically satisfied if \eqref{eq.fgrowth} is satisfied.
\end{remark}

\begin{remark}
    As we have mentioned in the introduction, V\u{a}rv\u{a}ruc\u{a} and Weiss emphasized that a particularly difficult regime occurs when the vorticity vanishes at the free surface and may have infinitely many sign changes accumulating there. 
    
    In the sign-changing case, their frequency analysis in Theorem \ref{thm.VWflat} is carried out under the condition \eqref{eq.fgrowth}. Assumption \ref{assumption:nonlinearity} contains this oscillatory linear-growth class. More generally, it controls the two signs in an asymmetrical way: the positive part remains linearly bounded, whereas the negative part may exhibit the logarithmic growth \(z(1+|\log z|)\), subject to the condition \eqref{eq.W0}.
\end{remark}

\begin{example}
    Let \(f(z)=-z|\log z|\) for \(z\in(0,z_0)\) with \(z_0<1\). A direct calculation shows that \(F(z)=\frac12 z^2\log z-\frac14 z^2\) and \(zf(z)-2F(z)=\frac 12 z^2\). Hence Assumption \ref{assumption:nonlinearity} holds with \(\mathsf P_0=0\), \(\mathsf N_0=1\) and \(\mathsf W_0=1/2\). On the other hand, 
    \[
        \frac{|f(z)|}{z} = |\log z|\to+\infty\qquad\text{ as }z\to 0^+,
    \] 
    so \(f(z)\) does not satisfy \eqref{eq.fgrowth} in Theorem \ref{thm.VWflat}. 

    Moreover, it should be noted that \(f(z)<0\) in every right neighborhood of \(0\). Consequently, \(\Delta u=-f(u)>0\) near the stagnation points in the positive phase, and the superharmonic boundary point principle of Oddson \cite{Odd68} used in \cite[Proposition 6.5]{VW12} is not applicable. Thus this example is covered by neither of the two alternatives used in \cite{VW12} to exclude horizontally flat singularities.
\end{example}

\begin{example}
    A more general nonlinearity can be given as 
    \[
        f_a(z)=z\bigl(
            a+\log z
        \bigr),\qquad 0<z<z_0<1,\qquad a>0.
    \]
    A direct computation shows that Assumption \ref{assumption:nonlinearity} holds with \(\mathsf P_0=a\), \(\mathsf N_0=1\) and \(\mathsf W_0=1/2\). If in particular \(e^{-a}<z_0\), then \(f_a(z)\) changes sign at \(z=e^{-a}\). As in the preceding example, 
    \[
        \frac{|f_a(z)|}{z} = |a+\log z|\to+\infty\qquad\text{ as }z\to 0^+,
    \]
    and hence \(f_a(z)\) does not satisfy \eqref{eq.fgrowth}. 
\end{example}

We now state our main results.

\begin{customthm}{I}\label{thm:main}
    \itshape 
    Let \(u\) be a variational solution (cf. Definition \ref{def:variational_solution}) to the problem \eqref{eq.wwp} and let \(x^0\in \Sigma^u\). Suppose in addition that the nonlinearity \(f(z)\) satisfies Assumption \ref{assumption:nonlinearity}. Then the limit \(D(0^+) := \lim_{r \to 0^+} D(r)\) exists and satisfies \(D(0^+) \ge 3/2\). Consequently, the blow-up sequence \(\{\phi_r\}_{r>0}\) defined in \eqref{eq.phir} is bounded in \(W^{1,2}(B_1)\) with 
    \[
        \int_{B_1}|\nabla\phi_r|^2\,\dx{x} + \int_{B_1}\phi_r^2\,\dx{x} \le C_nD(0^+),
    \]
    for all \(r>0\) sufficiently small, where \(C_n\) is a finite dimensional constant depends on \(x^0\).

    Moreover if \(n=2\), the convergence \(\phi_r \to \phi_0\) is strong in \(W^{1,2}_{\mathrm{loc}}(B_1\setminus\{0\})\) and the blow-up limit \(\phi_0\) is of the form
    \[
        \phi_0(x) =  
        \sqrt{\frac{2}{\pi}}\,\rho^{N(x^0)}
        \left|\sin\left(N(x^0)\min\{\max\{\theta,0\},\pi\}\right)\right|,
    \]
    where \(N(x^0)=D(0^+)\ge 2\) is an \emph{integer} that depends on \(x^0\).
\end{customthm}

\begin{remark}
    In comparison with Theorem \ref{thm.VWflat} from \cite{VW12}, our Theorem \ref{thm:main} yields the same compactness and convergence conclusions under substantially weaker assumptions: the growth condition \eqref{eq.fgrowth} is relaxed to the more general Assumption \ref{assumption:nonlinearity}. More notably, our proof strategy is entirely distinct from that of \cite{VW12}: we develop a new argument that completely avoids reliance on Bessel-type differential inequalities, which were a key technical cornerstone of the original proof. 
\end{remark}

As a direct application of Theorem \ref{thm:main}, we can rule out the existence of horizontal flat singularities for a larger class of nonlinearities, and consequently prove the Stokes conjecture for rotational waves with such vorticity.

\begin{customthm}{II}[cf. Main Theorem in \cite{VW12}]\label{thm:main2}
    \itshape
    Let \(n=2\) and let \(\psi\) be a weak solution (cf. Definition \ref{def:weak-solutions}) of 
    \begin{equation*}
        \left\{
            \begin{aligned}
                \Delta \psi(x,y)&=-f(\psi)\qquad
                &&\text{ in } \Omega\cap\{\psi>0\},\\[5pt]
                |\nabla \psi(x,y)|^2&=-y\qquad
                &&\text{ on } \Omega\cap\partial\{\psi>0\},
            \end{aligned}
        \right.
    \end{equation*}
    Assume that \(\psi\) satisfies 
    \[
        |\nabla \psi(x,y)|^2 \le C\max(-y,0)\quad\text{ locally in }\Omega,
    \]
    let the free boundary \(\partial\{\psi>0\}\) be a continuous injective curve \(\sigma=(\sigma_1,\sigma_2)\) such that \(\sigma(0)=(x_0,0)\), and assume that the vorticity function \(f\) satisfies either Assumption \ref{assumption:nonlinearity} or \(f(z)\ge 0\) for all \(z\) in a right neighborhood of \(0\). Then 
    \begin{enumerate}[label=(\roman*)]
        \item \label{item1} If the Lebesgue density of the set \(\{\psi>0\}\) at \((x_0,0)\) is positive, then the free boundary in a neighborhood of \((x_0,0)\) is the union of two \(C^1\)-graphs of functions \(\eta_1:(x_0-\delta,x_0]\to\R\) and \(\eta_2:[x_0,x_0+\delta)\to\R\) which are both continuously differentiable up to \(x_0\) and satisfy \(\eta_1'(x_0)=1/\sqrt{3}\) and \(\eta_2'(x_0)=-1/\sqrt{3}\);
        \item If the Lebesgue density of the set \(\{\psi>0\}\) at \((x_0,0)\) is zero, then \(\sigma_1(t)\ne x_0\) for \((-t_1,t_1)\setminus\{0\}\), \(\sigma_1(t)-x_0\) does not change its sign at \(t=0\) and 
        \[
            \lim_{t\to 0}\frac{\sigma_2(t)}{\sigma_1(t)-x_0}=0
        \]
    \end{enumerate}
    If we assume in addition that either \(\{\psi>0\}\) is a subgraph of a function in the \(y\)-direction or that \(\{\psi>0\}\) is a Lipschitz set, then the set of stagnation points is locally finite in \(\Omega\), and at each stagnation point \((x_0,0)\) the statement \ref{item1} holds.
\end{customthm}

We also state our new modified Weiss-type monotonicity formula and modified Almgren-type frequency formula, which are of independent interest. We believe theses newly inventioned tools have the potential to be applied to other free boundary problems with nonlinearities that exhibit logarithmic growth.

\begin{proposition}[Modified Weiss-type monotonicity formula]\label{prop:modified_Weiss}
    Let \(u\) be a variational solution of \eqref{eq.wwp}, let \(x^0\in \Sigma^u\) and let \(\delta:=\dist(x^0,\partial\Omega)/2\). Assume that the nonlinearity \(f(z)\) satisfies Assumption \ref{assumption:nonlinearity}. For each fixed \(\beta\in(0,1)\), define the modified Weiss-type functional
    \begin{equation}\label{eq:modified_Weiss}
        \Psi_\beta(r) = M(r) - M(0^+) + r^{-n-2+\beta}J(r),\qquad M(0^+):=\int_{B_1}x_n^+\,\dx{x},
    \end{equation}
    where \(M(r),J(r)\) are defined in \eqref{eq.Phir} and \eqref{eq.Jr_weiss}, respectively. Then there exists \(r_\beta\in(0,\delta)\) sufficiently small such that for a.e. \(r\in (0,r_\beta)\), 
    \begin{equation*}
        \frac{\dx{}}{\dx{r}}\bigl( e^{-\tfrac{2r^\beta}{\beta}}\Psi_\beta(r) \bigr)\ge 0.
    \end{equation*}
    Moreover, \(\Psi_\beta(r)\ge 0\) for all \(r\in (0,r_\beta)\) and \(\Psi_\beta(0^+)=0\).
\end{proposition}

Let \(x^0\in \Sigma^u\) and let \(\beta\in (0,1)\) be fixed as in Proposition \ref{prop:modified_Weiss}. We introduce the modified Almgren-type frequency function as

\begin{equation}\label{eq.Hbetar}
    H_\beta(r):=\frac{\Psi_\beta(r)}{r^{-n-2}J(r)} 
    =
    \frac{M(r) -M (0^+)}{r^{-n-2}J(r)} + r^\beta
    \ge 0.
\end{equation}

Define
\begin{equation}\label{eq.Vr}
    V(r):=\frac{r\int_{B_r(x^0)}x_n^+(1-\chi_{\{u>0\}})\,\dx{x}}{\int_{\partial B_r(x^0)} u^2\,\dH{n-1}}\ge 0,
\end{equation}

We establish the following frequency formula for \(H_\beta(r)\).

\begin{proposition}[Modified Almgren-type frequency formula, cf. Theorem 6.7 in \cite{VW12}]\label{prop:modified_Almgren}
    Under the same assumptions as in Proposition \ref{prop:modified_Weiss}, we have for a.e. \(r\in (0,r_\beta)\) the inequality
    \begin{equation}\label{eq.modified_Almgren}
        \begin{aligned}
            H_\beta'(r)
            &\ge
            \frac{2}{rJ(r)}
            \int_{\partial B_r(x^0)}
            \left[
            r\nabla u\cdot\nu-D(r)u
            \right]^2
            \,\dH{n-1}\\ 
            &\qquad
            +
            \frac2rV(r)H_\beta(r)
            +
            \frac2r\left(r^\beta-V(r)\right)^2.
        \end{aligned}
    \end{equation}
    In particular, \(H_\beta\) is nondecreasing on \((0,r_\beta)\), and the finite limit
    \[
        H(0^+):=\lim_{r\to0^+}H_\beta(r)
    \]
    exists and is independent of \(\beta\). 
\end{proposition}

\begin{remark}[Comparison with the frequency formula of V\u{a}rv\u{a}ruc\u{a} and Weiss]
    It is useful to compare our modified frequency formula \eqref{eq.modified_Almgren} with the frequency formula of V\u{a}rv\u{a}ruc\u{a} and Weiss. In \cite[Theorem 6.7]{VW12}, the frequency function is defined as \(H_{\rm VW}(r):=D(r)-V(r)\), and the identities \cite[(6.1)--(6.2)]{VW12} read, schematically,
    \begin{equation}\label{eq.VW_frequency}
        \begin{aligned}
            H_{\rm VW}'(r)
            &=
            \frac{2}{rJ(r)}
            \int_{\partial B_r(x^0)}
            \left[
            r\nabla u\cdot\nu-D(r)u
            \right]^2\,\dH{n-1} \\
            &\quad
            +
            \frac2rV(r)^2
            +
            \frac2rV(r)\left(H_{\rm VW}(r)-\frac32\right)
            +
            \frac{K(r)}{J(r)},
        \end{aligned}\tag{VW}
    \end{equation}
    where \(K(r)\) is defined in \eqref{eq.Kr}, and equivalently, writing \(D(r)=H_{\rm VW}(r)+V(r)\) in the first bracket,
    \[
        \begin{aligned}
        H_{\rm VW}'(r)
        &=
        \frac{2}{rJ(r)}
        \int_{\partial B_r(x^0)}
        \left[
        r\nabla u\cdot\nu-H_{\rm VW}(r)u
        \right]^2\,\dH{n-1}  \\
        &\quad
        +
        \frac2rV(r)\left(H_{\rm VW}(r)-\frac32\right)
        +
        \frac{K(r)}{J(r)}.
        \end{aligned}
    \]
    Thus the Vărvărucă--Weiss formula is an exact \emph{identity} for the unmodified frequency \(H_{\rm VW}=D-V\). Its right-hand side still contains the nonlinear error \(K(r)/J(r)\), whose sign is not controlled in general.

    Our frequency formula is different. We work with the modified quantity
    \[
        H_\beta(r)
        :=
        \frac{\Psi_\beta(r)}{r^{-n-2}J(r)}
        =
        D(r)-\frac32-V(r)+r^\beta
        =
        H_{\rm VW}(r)-\frac32+r^\beta.
    \]
    The shift by \(3/2\) subtracts the full-density homogeneous scale, while the added term \(r^\beta\) comes from the boundary correction in the modified Weiss functional \eqref{eq:modified_Weiss}. With this correction, the nonlinear error term \(K(r)/J(r)\) is absorbed before taking the quotient. The resulting formula is the one-sided differential inequality \eqref{eq.modified_Almgren}. Every term on the right-hand side of \eqref{eq.modified_Almgren} is nonnegative, since \(V(r)\ge0\) and \(H_\beta(r)\ge0\). 
\end{remark}

\begin{remark}
    Let us also remark that the modified frequency formula \eqref{eq.modified_Almgren} cannot be obtained from \cite[(6.1)--(6.2)]{VW12}. If one substitutes \(H_{\rm VW}=H_\beta+\frac32-r^\beta\) into \eqref{eq.VW_frequency}, the term \(K(r)/J(r)\) remains. In \cite{VW12}, this term is controlled by the assumption \eqref{eq.fgrowth}, which yields \(|F(s)|\le Cs^2\) and hence the key estimate \eqref{eq.esmVW}. This estimate is then used to derive the Bessel-type differential inequality in \cite[Theorem 6.12]{VW12}.

    Under the present assumptions this route is unavailable. The negative part of \(f\) satisfies \(f^-(z)\le C z(1+|\log z|)\). Thus \(K(r)\) need not satisfy the estimate above. The correction \(r^{-n-2+\beta} J(r)\) then converts this one-sided control into the positive square \((r^\beta-V(r))^2\) in the modified frequency formula. This is the point at which the present argument differs essentially from the frequency method in \cite{VW12}.
\end{remark}

\subsection{Organization of the paper}
Our paper is organized as follows. In Section \ref{sec:preliminaries}, we introduce some notations and notions of solutions, and review some preliminary results from \cite{VW12}. In Section \ref{sec:modified_Weiss}, we prove the modified Weiss-type monotonicity formula in Proposition \ref{prop:modified_Weiss}. In Section \ref{sec:modified_Almgren}, we establish the modified Almgren-type frequency formula in Proposition \ref{prop:modified_Almgren}. Finally, in Section \ref{sec:strong_convergence}, we prove the strong convergence of the frequency sequence and complete the proof of Theorem \ref{thm:main}.

\subsection{Related works and historical notes}

If \(f \equiv 0\), the problem \eqref{eq.wwp} reduces to the classical one-phase Bernoulli-type free boundary problem with weight \(x_n\). The nondegenerate weight case (i.e., \(f \equiv 0\) and \(|\nabla u|^2 = 1\)) has been extensively studied in the literature; see, for instance, \cite{AC81,Wei99,CJK05,DJ09,JS15,EE19,Vel23,KW25} and the references therein. The degenerate weight case with zero right-hand side (\(f \equiv 0\)) was first studied by V\u{a}rv\u{a}ruc\u{a} and Weiss \cite{VW12} in the context of the \emph{Stokes conjecture} for two-dimensional irrotational gravity water waves. This conjecture was proposed by G. Stokes in \cite{Sto80} in 1880. It states that at any stagnation point the free surface forms a (symmetric) corner of \(120^\circ\). The conjecture was proved by Amick, Fraenkel and Toland \cite{AFT82} and by Plotnikov \cite{Plo08} under additional structural assumptions, while a new geometric proof without such assumptions was given by V\u{a}rv\u{a}ruc\u{a} and Weiss in \cite{VW11}. Their approach is also applicable to three-dimensional axisymmetric irrotational flows \cite{VW14}.

If \(f \not\equiv 0\), the problem \eqref{eq.wwp} reduces to a semilinear Bernoulli-type free boundary problem, which is typically used to describe rotational flows. The mathematical theory for this setting is now well developed and contains a wealth of results. The existence of global periodic traveling waves of finite depth was proved by Constantin and Strauss \cite{CS04}, while the existence of infinite-depth waves was established by Hur \cite{Hur06}. The waves obtained in \cite{CS04,Hur06} are monotone and symmetric in each period, and under certain conditions \cite{Var08}, the values of \(|\nabla u|\) at the stagnation points can be arbitrarily close to zero. It is therefore expected that rotational waves also form a \(120^\circ\) corner at the stagnation points for many vorticity distributions. This is usually referred to as the \emph{Stokes conjecture} for rotational waves. 

The first rigorous result in this direction was obtained by V\u{a}rv\u{a}ruc\u{a} \cite{Var09}. In \cite[Theorem 5.2]{Var09}, he proved the following dichotomy: at a stagnation point where the profile is symmetric and locally monotone, the wave profile must have \textit{either} a \(120^\circ\) corner \textit{or} a horizontal tangent. Crucially, V\u{a}rv\u{a}ruc\u{a} showed that if the vorticity is non-negative near the free surface (\(f(z) \ge 0\) for \(z \in [0,z_0]\)), then the horizontal tangent alternative is excluded by the maximum principle. Thus, positive vorticity near the crest guarantees that the \(120^\circ\) Stokes conjecture holds. However, if the vorticity is strictly negative near the free surface (\(f(z) \le 0\) for \(z \in [0,z_0]\)), the maximum principle cannot rule out the horizontal tangent alternative. Moreover, the global existence of extreme waves via global bifurcation in \cite{Var09} is proved precisely under the restrictive assumption of globally negative vorticity (\(f(z) \le 0\) for all \(z \in [0,z_0]\) and \(f(0) < 0\)). It should be noted that V\u{a}rv\u{a}ruc\u{a}'s analysis in \cite{Var09} also relies on a priori structural assumptions on the free surface, as in \cite{AFT82,Plo08}.

The Stokes conjecture for rotational flows without any structural assumptions was established by V\u{a}rv\u{a}ruc\u{a} and Weiss in \cite{VW12}. Their results (cf.\ the main theorems in \cite{VW12}) yield the following dichotomy at stagnation points for rotational waves: either the wave profile has a \(120^\circ\) corner or it is a cusp point. The present work is based on the variational method developed by V\u{a}rv\u{a}ruc\u{a} and Weiss in \cite{VW12}, and we extend their results to a larger class of nonlinearities, which includes the model example \(f(z)=-z|\log z|\). 

\section{Notations and notions of solutions}\label{sec:preliminaries}

We denote by \(x\cdot y\) the Euclidean inner product in \(\R^n\times\R^n\), by \(B_r(x^0):=\{x\in\R^n:|x-x^0|<r\}\) the ball of radius \(r\) centered at \(x^0\). We abbreviate \(B_r(0)\) as \(B_r\), and denote by \(\omega_n\) the \(n\)-dimensional volume of \(B_1\). \(\mathcal{L}^n\) and \(\mathcal{H}^{n-1}\) denote the \(n\)-dimensional Lebesgue measure and the \((n-1)\)-dimensional Hausdorff measure, respectively. For a set \(E\subset\R^n\), we write \(\chi_E\) for its characteristic function.

\begin{definition}[Variational solutions]\label{def:variational_solution}
    We say that \(u\in W_{{\rm loc}}^{1,2}(\Omega)\) is a variational solution to the problem \eqref{eq.wwp} if 
    \begin{enumerate}[label=(\roman*)]
        \item \(u\in C^0(\Omega)\cap C^2(\Omega\cap\{u>0\})\), \(u\ge 0\) in \(\Omega\) and \(u\equiv 0\) in \(\Omega\cap\{x_n\le 0\}\); 
        \item For any test function \(\phi\in C_c^1(\Omega;\R^n)\), we have
        \begin{equation*}
            0=\int_\Omega \left[ \left(|\nabla u|^2-2F(u)\right)\operatorname{div}\phi - 2\nabla u D\phi \nabla u + x_n \chi_{\{u>0\}}\operatorname{div}\phi + \chi_{\{u>0\}}\phi_n \right]\,\dx{x}.
        \end{equation*} 
    \end{enumerate}
\end{definition}

We also recall the definition of weak solutions to the problem \eqref{eq.wwp}.

\begin{definition}[Weak solutions]\label{def:weak-solutions}
    We say that \(u\in W_{{\rm loc}}^{1,2}(\Omega)\cap C^2(\Omega\cap\{u>0\})\) is a weak solution of \eqref{eq.wwp} provided that 
    \begin{enumerate}[label=(\roman*)]
        \item \(u\) is a variational solution of \eqref{eq.wwp};
        \item The topological free boundary \(\partial\{u>0\}\cap\Omega\cap\{x_n>0\}\) is locally a \(C^{2,\alpha}\)-hypersurface.
    \end{enumerate}
\end{definition} 

\begin{remark}
    For any weak solution \(u\) of \eqref{eq.wwp} such that the growth condition \eqref{eq.growth_assumption} holds, \(u\) is a variational solution in the sense of Definition \ref{def:variational_solution}. Moreover, \(\chi_{\{u>0\}}\) is locally a function of bounded variation, and the total variation measure \(|\nabla\chi_{\{u>0\}}|\) satisfies 
    \[
        r^{1/2-n}\int_{B_r(y)}\sqrt{x_n}\,\dx{}|\nabla\chi_{\{u>0\}}|\le C_0
    \]
    for all \(B_r(y)\Subset\Omega\) such that \(y_n=0\) (see \cite[Lemma 3.4]{VW11}). Here for a set \(E\) of locally finite perimeter, \(|\nabla\chi_E|=\cH^{n-1}\mres\partial^*E\). In particular, if \(E\) has a \(C^1\)-boundary, this is \(\mathcal H^{n-1}\mres\partial E\).
\end{remark}

\section{The modified Weiss-type monotonicity formula}
\label{sec:modified_Weiss}

In this section we prove Proposition \ref{prop:modified_Weiss}. Throughout the section, let \(u\) be a variational solution of \eqref{eq.wwp} in the sense of Definition \ref{def:variational_solution}, let \(0\in\Sigma^u\), and assume that \(f(z)\) satisfies Assumption \ref{assumption:nonlinearity}. Moreover, it follows from the growth condition \eqref{eq.growth_assumption} that
\[
    0\le u(x)\le C|x|^{3/2}\qquad\text{ for }x\in\Omega\text{ sufficiently close to }0.
\]
Hence, after decreasing the radius if necessary, we may assume that
\begin{equation}\label{eq.smallness_u}
    0\le u(x)<z_0
    \qquad\text{for }x\in B_\delta,
\end{equation}
where \(\delta_0:=\frac12\operatorname{dist}(0,\partial\Omega)\). We may choose \(\delta_0\) small enough so that 
\begin{equation}\label{eq.choice_delta}
    \delta_0^2\mathsf P_0\le \frac n2,
\end{equation}
where \(\mathsf P_0\) is defined in \eqref{eq.P0N0}. In \cite[Corollary 6.11]{VW12}, a stronger version of the forthcoming lemma is proved under the assumption \(|f(z)|\le Cz\). We here provide a self-contained proof for the sake of completeness and to show that the weaker Assumption \ref{assumption:nonlinearity} still suffices to provide a one-sided lower bound for \(K(r)\) in terms of \(J(r)\) in \eqref{eq.Jr_weiss}. This is crucial for the proof of Proposition \ref{prop:modified_Weiss}.

\begin{lemma}\label{lem:boundary_control_Weiss_defect}
    Let \(J(r)\) be defined as in \eqref{eq.Jr_weiss}, then for a.e. \(r\in(0,\delta_0)\),
    \begin{equation}\label{eq:L2_boundary_control}
        \int_{B_r}u^2\,\dx{x}
        \le
        \frac2n rJ(r),
    \end{equation}
    and
    \begin{equation}\label{eq:K_lower_bound_autonomous}
        r^{-n-2}K(r)
        \ge
        -\left[
            \mathsf W_0+\frac 2n\bigl(
                2\mathsf P_0+(n-2)\mathsf W_0
            \bigr)
        \right]\,r^{-n-1}J(r).
    \end{equation}
    Here \(K(r)\) is defined in \eqref{eq.Kr}.
\end{lemma}

\begin{proof}
    We first prove \eqref{eq:L2_boundary_control}. It follows from Lemma 6.9 in \cite{VW12} that for a.e. \(r\in(0,\delta_0)\),
    \[
        rJ(r)
        =
        \int_{B_r}
        \left[
            nu^2+
            \bigl(|\nabla u|^2-u f(u)\bigr)
            \bigl(r^2-|x|^2\bigr)
        \right]\,\dx{x}.
    \]
    Since \(u f(u)=u f^+(u)-u f^-(u)\le u f^+(u)\), \eqref{eq.smallness_u} and \eqref{eq.P0N0} in Assumption \ref{assumption:nonlinearity} imply that \(u f^+(u)\le \mathsf P_0 u^2\) in \(B_\delta\). Using also \(|\nabla u|^2\ge0\) and \(0\le r^2-|x|^2\le r^2\), we obtain
    \[
        rJ(r)
        \ge
        \int_{B_r}
        \left[
            nu^2-\mathsf P_0 u^2
            \bigl(r^2-|x|^2\bigr)
        \right]\,\dx{x}  
        \ge
        \bigl(
            n-r^2\,\mathsf P_0
        \bigr)
        \int_{B_r}u^2\,\dx{x}.
    \]
    By the choice of \(\delta_0\) which satisfies \eqref{eq.choice_delta}, we have \(r^2\mathsf P_0\le n/2\) for \(0<r<\delta_0\). Therefore \eqref{eq:L2_boundary_control} is proved.

    We now prove \eqref{eq:K_lower_bound_autonomous}. First, it follows from \eqref{eq.W0} in Assumption \ref{assumption:nonlinearity} that
    \begin{equation}\label{eq:K-integrand-lower-bounds-1}
        2F(z)-zf(z)\ge -[2F(z)-zf(z)]^-\ge -\mathsf W_0\,z^2,
    \end{equation}
    and from \eqref{eq.P0N0} in Assumption \ref{assumption:nonlinearity} that
    \begin{equation}\label{eq:K-integrand-lower-bounds-2}
        (n-2)zf(z)-2nF(z)\ge-\bigl[
            (n-2)zf(z)-2nF(z)
        \bigr]^-\ge -\bigl(
            2\mathsf P_0+(n-2)\mathsf W_0
        \bigr)z^2,
    \end{equation}
    for all \(0<z<z_0\). Applying \eqref{eq:K-integrand-lower-bounds-1} and \eqref{eq:K-integrand-lower-bounds-2} with \(z=u(x)\) and using \eqref{eq.smallness_u}, we deduce from the definition of \(K(r)\) in \eqref{eq.Kr} that for a.e. \(r\in(0,\delta_0)\),
    \[
        K(r)\ge-\mathsf W_0\, r\int_{\partial B_r}u^2\,\dH{n-1}-(2\mathsf P_0+(n-2)\mathsf W_0)\int_{B_r}u^2\,\dx{x}.
    \]
    Using \eqref{eq:L2_boundary_control}, we get
    \[
        K(r)
        \ge
        -\left[
            \mathsf W_0+\frac 2n\bigl(
                2\mathsf P_0+(n-2)\mathsf W_0
            \bigr)
        \right]rJ(r).
    \]
    Dividing by \(r^{n+2}\), we obtain \eqref{eq:K_lower_bound_autonomous}.
\end{proof}
We next introduce an identity that will be used in the proof of Proposition \ref{prop:modified_Weiss}. 

\begin{lemma}\label{lem:boundary_correction_derivative}
    Let \(J(r)\) be defined as in \eqref{eq.Jr_weiss}. Then for a.e. \(r\in(0,\delta_0)\), one has
    \begin{equation}\label{eq:boundary_correction_derivative}
    \begin{aligned}
        \frac{\dx{}}{\dx r}\left(r^{-n-2}J(r)\right)
        =
        \frac2r
        \bigg[
            M(r)-M(0^+)
            +
            r^{-n-1}
            \int_{B_r}
            x_n^+
            \bigl(1-\chi_{\{u>0\}}\bigr)\,\dx{x}
        \bigg].
    \end{aligned}
    \end{equation}
    Moreover,
    \begin{equation}\label{eq:Psi_zero_limit}
        \lim_{r\to0^+}\Psi_\beta(r)=0.
    \end{equation}
\end{lemma}

\begin{proof}
    Assume without loss of generality that \(x^0=0\), a direct computation gives
    \begin{equation}\label{eq:scaled_J_derivative_raw}
    \begin{aligned}
        \frac{\dx{}}{\dx r}\left(r^{-n-2}J(r)\right)
        &=
        r^{-n-3}
        \left[
        2r\int_{\partial B_r}u\,\nabla u\cdot\nu\,\dH{n-1}
        -
        3J(r)
        \right]\\
        &=
        \frac2r
        \left[
            r^{-n-1}
            \int_{\partial B_r}u\,\nabla u\cdot\nu\,\dH{n-1}
            -
            \frac32 r^{-n-2}J(r)
        \right].
    \end{aligned}
    \end{equation}
    Using the energy identity (cf. \cite[equation (3.10)]{VW12}),
    \begin{equation}\label{eq.energy_identity}
        \int_{\partial B_r}u\,\nabla u\cdot\nu\,\dH{n-1}
        =
        \int_{B_r}
        \bigl(|\nabla u|^2-u f(u)\bigr)\,\dx{x}.
    \end{equation}
    The equation \eqref{eq:scaled_J_derivative_raw} becomes
    \begin{equation}\label{eq:scaled_J_derivative_energy}
        \begin{aligned}
            \frac{\dx{}}{\dx r}\left(r^{-n-2}J(r)\right)
            =
            \frac2r
            \left[
                r^{-n-1}
                \int_{B_r}
                \bigl(|\nabla u|^2-u f(u)\bigr)\,\dx{x}
                -
                \frac32 r^{-n-2}J(r)
            \right].
        \end{aligned}
    \end{equation}

    Since \(x^0\in\Sigma^u\), we have 
    \[
        M(0^+)=\int_{B_1}x_n^+\,\dx{x}=r^{-n-1}\int_{B_r}x_n^+\,\dx{x}.
    \]
    Using the definition of \(M(r)\) in \eqref{eq.Phir}, we therefore obtain
    \begin{equation}\label{eq:M_minus_M0_identity}
        \begin{aligned}
            M(r)-M(0^+)
            &=
            r^{-n-1}
            \int_{B_r}
            \bigl(|\nabla u|^2-u f(u)\bigr)\,\dx{x}
            -
            \frac32 r^{-n-2}J(r) \\
            &\quad
            -
            r^{-n-1}
            \int_{B_r}
            x_n^+
            \bigl(1-\chi_{\{u>0\}}\bigr)\,\dx{x}.
        \end{aligned}
    \end{equation}
    Combining \eqref{eq:scaled_J_derivative_energy} and
    \eqref{eq:M_minus_M0_identity} proves
    \eqref{eq:boundary_correction_derivative}.

    It remains to prove \eqref{eq:Psi_zero_limit}. By the definition of \(M(0^+)\), we have \(\lim_{r\to0^+}\bigl(M(r)-M(0^+)\bigr)=0\). Moreover, the growth condition \eqref{eq.growth_assumption} gives \(J(r)=\int_{\partial B_r}u^2\le Cr^{n+2}\). Thus \(r^{-n-2+\beta}J(r)\to 0\) as \(r\to 0^+\). Using the definition \eqref{eq:modified_Weiss}, we conclude that the limit \(\lim_{r\to0^+}\Psi_\beta(r)=0\) and \eqref{eq:Psi_zero_limit} holds.
\end{proof}

The following lemma is the key step in proving the monotonicity of \(\Psi_\beta(r)\).

\begin{lemma}[Differential inequality for \(\Psi_\beta\)]\label{lem:Psi_beta_differential_inequality}
    For every \(\beta\in(0,1)\), there exists \(r_\beta\in(0,\delta_0)\) such that, for a.e.
    \(r\in(0,r_\beta)\),
    \begin{equation}\label{eq:Psi_beta_differential_inequality}
        \begin{aligned}
            \Psi_\beta'(r)
            \ge
            2r^{\beta-1}\Psi_\beta(r)
            +
            2r^{-n-1}
            \int_{\partial B_r}
            \left(
                \nabla u\cdot\nu-\frac{3}{2}\frac{u}{r}
            \right)^2
            \,\dH{n-1}.
        \end{aligned}
    \end{equation}
\end{lemma}

\begin{proof}
    Recall the following monotonicity formula \eqref{mono.Wei.model1} in Theorem \ref{thm:monotonicity_formula},
    \[
        M'(r)
        =
        2r^{-n-1}
        \int_{\partial B_r}
        \left(
            \nabla u\cdot\nu-\frac{3}{2}\frac{u}{r}
        \right)^2
        \,\dH{n-1}
        +
        r^{-n-2}K(r)
    \]
    for a.e. \(r\in(0,\delta_0)\). Using Lemma \ref{lem:boundary_control_Weiss_defect}, we overestimate the last additional term to obtain
    \begin{equation}\label{eq:M_prime_lower_bound}
        M'(r)
        \ge
        2r^{-n-1}
        \int_{\partial B_r}
        \left(
            \nabla u\cdot\nu-\frac{3}{2}\frac{u}{r}
        \right)^2
        \,\dH{n-1}
        -
        \mathsf W\,r\bigl(r^{-n-2}J(r)\bigr),
    \end{equation}
    where \(\mathsf W:=[\mathsf W_0+\frac 2n\bigl(2\mathsf P_0+(n-1)\mathsf W_0\bigr)]\) is the constant defined as in \eqref{eq:K_lower_bound_autonomous}.

    Recall \eqref{eq:modified_Weiss}, differentiating it with respect to \(r\) and using Lemma \ref{lem:boundary_correction_derivative}, we find
    \begin{equation}\label{eq:Psi_prime_expand}
        \begin{aligned}
            \Psi_\beta'(r)
            &=
            M'(r)
            +
            \beta r^{\beta-1}\bigl(r^{-n-2}J(r)\bigr)
            +
            r^\beta
            \frac{\dx{}}{\dx r}
            \left(r^{-n-2}J(r)\right)  \\
            &=
            M'(r)
            +
            \beta r^{\beta-1}\bigl(r^{-n-2}J(r)\bigr) \\
            &\quad
            +
            2r^{\beta-1}
            \bigg[
                M(r)-M(0^+)
                +
                r^{-n-1}
                \int_{B_r}
                x_n^+
                \bigl(1-\chi_{\{u>0\}}\bigr)\,\dx{x}
            \bigg].
        \end{aligned}
    \end{equation}
    Since \(M(r)-M(0^+)=\Psi_\beta(r)-r^\beta\bigl(r^{-n-2}J(r)\bigr)\), we deduce from \eqref{eq:M_prime_lower_bound} and \eqref{eq:Psi_prime_expand} that
    \begin{equation}\label{eq:Psi_prime_before_absorption}
        \begin{aligned}
            \Psi_\beta'(r)
            &\ge
            2r^{-n-1}
            \int_{\partial B_r}
            \left(
                \nabla u\cdot\nu-\frac{3}{2}\frac{u}{r}
            \right)^2
            \,\dH{n-1}
            +
            2r^{\beta-1}\Psi_\beta(r) \\
            &\quad
            +
            2r^{\beta-1}
            r^{-n-1}
            \int_{B_r}
            x_n^+
            \bigl(1-\chi_{\{u>0\}}\bigr)\,\dx{x} \\
            &\quad
            +
            \left(
                \beta r^{\beta-1}
                -
                2r^{2\beta-1}
                -
                \mathsf W\, r
            \right)
            \bigl(r^{-n-2}J(r)\bigr).
        \end{aligned}
    \end{equation}
    The third term on the right-hand side is nonnegative. Moreover, we can rewrite the terms of the last term in the bracket as
    \[
        \beta r^{\beta-1}
        -
        2r^{2\beta-1}
        -
        \mathsf W\, r
        =
        r^{\beta-1}
        \left(
            \beta
            -
            2r^\beta
            -
            \mathsf W\, r^{2-\beta}
        \right).
    \]
    Since \(\beta\in(0,1)\), we may choose \(r_\beta\in(0,\delta_0)\) so small that \(2r^\beta+\mathsf W r^{2-\beta}\le\beta\) for \(0<r<r_\beta\). Therefore the last term on the right hand side in \eqref{eq:Psi_prime_before_absorption} is nonnegative for \(0<r<r_\beta\). This proves \eqref{eq:Psi_beta_differential_inequality}.
\end{proof}

We now prove Proposition \ref{prop:modified_Weiss}.
\begin{proof}[Proof of Proposition \ref{prop:modified_Weiss}]
    By Lemma \ref{lem:Psi_beta_differential_inequality}, for a.e.
    \(r\in(0,r_\beta)\),
    \[
        \Psi_\beta'(r)-2r^{\beta-1}\Psi_\beta(r)
        \ge
        2r^{-n-1}
        \int_{\partial B_r}
        \left(
            \nabla u\cdot\nu-\frac{3}{2}\frac{u}{r}
        \right)^2
        \,\dH{n-1}
        \ge0.
    \]
    Multiplying by \(e^{-\frac{2r^\beta}{\beta}}\), we obtain
    \[
        \frac{\dx{}}{\dx r}
        \left(
            e^{-\frac{2r^\beta}{\beta}}\Psi_\beta(r)
        \right)
        =
        e^{-\frac{2r^\beta}{\beta}}
        \left[
            \Psi_\beta'(r)-2r^{\beta-1}\Psi_\beta(r)
        \right]
        \ge0
    \]
    for a.e. \(r\in(0,r_\beta)\). Hence \(r\mapsto e^{-\frac{2r^\beta}{\beta}}\Psi_\beta(r)\) is nondecreasing on \((0,r_\beta)\). By Lemma \ref{lem:boundary_correction_derivative}, \(\lim_{r\to0^+}\Psi_\beta(r)=0\). Therefore, for \(0<r<r_\beta\),
    \[
        e^{-\frac{2r^\beta}{\beta}}\Psi_\beta(r)
        \ge
        \lim_{t\to0^+}
        e^{-\frac{2t^\beta}{\beta}}\Psi_\beta(t)
        =
        0.
    \]
    Since the exponential factor is positive, we obtain \(\Psi_\beta(r)\ge0\) for \(0<r<r_\beta\). The identity \(\Psi_\beta(0^+)=0\) is exactly \eqref{eq:Psi_zero_limit}. This concludes the proof.
\end{proof} 

\section{The modified Almgren-type frequency formula}
\label{sec:modified_Almgren}

Throughout this section, let \(u\) be a variational solution to the problem \eqref{eq.wwp}, let \(0\in\Sigma^u\) and assume that \(f\) satisfies Assumption \ref{assumption:nonlinearity}. Recall \(D(r)\), \(V(r)\), \(H_\beta(r)\) and \(\Psi_\beta(r)\) defined in \eqref{eq.freqDr}, \eqref{eq.Vr}, \eqref{eq.Hbetar} and \eqref{eq:modified_Weiss}, respectively. We also define 
\begin{equation*}
    E_{x^0,u}(r)=E(r)=
    2r^{-n-1}\int_{\partial B_r}
    \left(
        \nabla u\cdot\nu 
        -
        \frac 32 \frac{u}{r}
    \right)^2\,\dH{n-1},
\end{equation*}
which is the square term in \eqref{mono.Wei.model1}. Then the differential inequality obtained in Lemma \ref{lem:Psi_beta_differential_inequality} (cf. \eqref{eq:Psi_beta_differential_inequality}) can be rewritten as 
\begin{equation}\label{eq:Psi_differential_from_Weiss}
    \Psi_\beta'(r)
    \ge 
    2r^{\beta-1}\Psi_\beta(r)
    +E(r).
\end{equation} 

Let \(\delta_0:=\tfrac 12\dist(0,\partial\Omega)\), then we summarize some basic identities that will be used in the proof of Proposition \ref{prop:modified_Almgren} in the following Lemma.

\begin{lemma}[Basic identities] 
    For a.e. \(r\in(0,\delta_0)\), the following identities hold: 
    \begin{equation}\label{eq:J_log_derivative} 
        \frac{(r^{-n-2}J(r))'}{r^{-n-2}J(r)} = \frac2r \left( D(r)-\frac32 \right), 
    \end{equation} 
    \begin{equation}\label{eq:D_H_V_identity} 
        D(r)-\frac32 = H_\beta(r)+V(r)-r^\beta, 
    \end{equation} and 
    \begin{equation}\label{eq:H_preliminary_inequality} 
        H_\beta'(r) 
        \ge 
        2r^{\beta-1}H_\beta(r) 
        +  \frac{E(r)}{r^{-n-2}J(r)} 
        - H_\beta(r) \frac{(r^{-n-2}J(r))'}{r^{-n-2}J(r)}. 
    \end{equation}
\end{lemma}

\begin{proof}
    The first identity follows by dividing the identity \eqref{eq:scaled_J_derivative_raw} by \(r^{-n-2}J(r)\) and the energy identity in \eqref{eq.energy_identity}. We next prove \eqref{eq:D_H_V_identity}, since \(x^0\in\Sigma^u\) we have \(M(0^+)=r^{-n-1}\int_{B_r}x_n^+\,\dx{x}\). Recalling the identities \eqref{eq:M_minus_M0_identity} and \eqref{eq.energy_identity}, we have 
    \[
        M(r)-M(0^+)=(r^{-n-2}J(r))\cdot\left( D(r) - \frac 32 \right) - r^{-n-1}\int_{B_r}x_n^+(1-\chi_{\{u>0\}})\,\dx{x}.
    \]
    Using the definitions of \(\Psi_\beta(r)\), \(H_\beta(r)\) and \(V(r)\), we obtain 
    \[
        H_\beta(r) = \frac{M(r)-M(0^+)}{r^{-n-2}J(r)} + r^\beta = D(r) - \frac 32 - V(r) + r^\beta,
    \]
    which is exactly \eqref{eq:D_H_V_identity}. Finally, differentiating \(H_\beta(r)\) and using \eqref{eq:Psi_differential_from_Weiss}, we obtain 
    \[ 
        \begin{aligned} 
            H_\beta'(r) &= \frac{\Psi_\beta'(r)}{r^{-n-2}J(r)} - H_\beta(r) \frac{(r^{-n-2}J(r))'}{r^{-n-2}J(r)} \\ 
            &\ge 2r^{\beta-1} \frac{\Psi_\beta(r)}{r^{-n-2}J(r)} + \frac{E(r)}{r^{-n-2}J(r)} - H_\beta(r) \frac{(r^{-n-2}J(r))'}{r^{-n-2}J(r)}  \\ 
            &= 2r^{\beta-1}H_\beta(r) +  \frac{E(r)}{r^{-n-2}J(r)} - H_\beta(r) \frac{(r^{-n-2}J(r))'}{r^{-n-2}J(r)}. 
        \end{aligned} 
    \]
    This proves \eqref{eq:H_preliminary_inequality}.
\end{proof}

\begin{proof}[Proof of Proposition \ref{prop:modified_Almgren}]
    We first rewrite the quotient \(E(r)/r^{-n-2}J(r)\) as
    \[ 
        \frac{E(r)}{r^{-n-2}J(r)} = \frac{2}{rJ(r)} \int_{\partial B_r} \left( 
            r\nabla u\cdot\nu-\frac32u 
        \right)^2 \,\dH{n-1}.
    \] 
    We expand 
    \[ 
        r\nabla u\cdot\nu-\frac32u 
        = 
        \left[ r\nabla u\cdot\nu-D(r)u \right] 
        + \left( D(r)-\frac32 \right)u.
    \] 
    Note that the cross term vanishes, since
    \[ 
        \int_{\partial B_r} u \left[ r\nabla u\cdot\nu-D(r)u \right] \,\dH{n-1} = r\int_{\partial B_r}u\nabla u\cdot\nu\,\dH{n-1} - D(r)J(r) =0.
    \] 
    Therefore 
    \begin{equation}\label{eq:E_over_J_radial_split} 
        \frac{E(r)}{r^{-n-2}J(r)} 
        = 
        \frac{2}{rJ(r)} \int_{\partial B_r} \left[ 
            r\nabla u\cdot \nu-D(r)u 
        \right]^2 \,\dH{n-1} + \frac2r \left( D(r)-\frac32 \right)^2 . 
    \end{equation} Combining \eqref{eq:H_preliminary_inequality}, \eqref{eq:J_log_derivative}, and \eqref{eq:E_over_J_radial_split}, we get 
    \[ 
        \begin{aligned} 
            H_\beta'(r) &\ge 
            \frac{2}{rJ(r)} \int_{\partial B_r} \left[ r\nabla u\cdot\nu-D(r)u \right]^2 \,\dH{n-1} \\ 
            &\quad 
            + \frac2r \left[ \left(D(r)-\frac32\right)^2 - \left(D(r)-\frac32\right)H_\beta(r) + r^\beta H_\beta(r) \right]. 
        \end{aligned} 
    \] 
    Using \eqref{eq:D_H_V_identity}, set \(Q(r):=D(r)-\frac32 = H_\beta(r)+V(r)-r^\beta \). Then a direct computation shows that 
    \[
        [Q(r)]^2-Q(r)H_\beta(r)+r^\beta H_\beta(r) = V(r)H_\beta(r) + \left( r^\beta-V(r) \right)^2. 
    \] 
    Substituting this into the previous inequality gives \eqref{eq.modified_Almgren}.
\end{proof}

As an application of the modified frequency formula \eqref{eq.modified_Almgren}, we have

\begin{corollary}\label{cor:weak_frequency_monotonicity}
    Fix \(\beta\in(0,1)\). Then there exists \(r_\beta>0\) sufficiently small such that for a.e. \(r\in(0,r_\beta)\), 
    \begin{equation}\label{eq:weak_frequency_formula} 
        H_\beta'(r) \ge \frac2rV(r)H_\beta(r) + \frac2r \left( r^\beta-V(r) \right)^2 . 
    \end{equation} 
    Consequently, \(H_\beta\) is nondecreasing on \((0,r_\beta)\), and the finite right limit 
    \[ 
        H_\beta(0^+):=\lim_{r\to0^+}H_\beta(r) 
    \] 
    exists. Moreover, this limit is independent of \(\beta\in(0,1)\) so we denote this limit simply by \(H(0^+)\).
    
    Finally, for every \(r\in(0,r_\beta)\), 
    \begin{equation}\label{eq:radial_defect_integrability} 
        \int_0^r \frac1{tJ(t)} \int_{\partial B_t} \left[ t\nabla u\cdot\nu-D(t)u \right]^2 \,\dH{n-1}\,\dx{t} <+\infty. 
    \end{equation} 
\end{corollary}

\begin{proof}
    The inequality \eqref{eq:weak_frequency_formula} follows immediately from \eqref{eq.modified_Almgren} by dropping the first nonnegative term. Since \(V(r)\ge0\) by definition and \(\Psi_\beta(r)\ge0\) by Proposition \ref{prop:modified_Weiss}, we have \(H_\beta(r)\ge0\). Therefore \(H_\beta\) is nondecreasing on \((0,r_\beta)\). Since \(H_\beta\ge0\), for every fixed \(r_1\in(0,r_\beta)\) we have \(0\le H_\beta(r)\le H_\beta(r_1)\) for \(0<r<r_1\). Thus the right limit \(H_\beta(0^+)\) exists and is finite. If \(\beta_1,\beta_2\in(0,1)\), then by \eqref{eq.Hbetar}, \(H_{\beta_1}(r)-H_{\beta_2}(r) = r^{\beta_1}-r^{\beta_2}\). Letting \(r\to0^+\), we obtain \(H_{\beta_1}(0^+)=H_{\beta_2}(0^+)\). We shall denote the common value by \(H(0^+)\), which is independent of \(\beta\). It remains to prove \eqref{eq:radial_defect_integrability}. It follows from \eqref{eq.modified_Almgren} that
    \[ 
        H_\beta'(t) \ge \frac{2}{tJ(t)} \int_{\partial B_t} \left[ t\nabla u\cdot\nu-D(t)u \right]^2 \,\dH{n-1} 
    \] 
    for a.e. \(t\in(0,r_1)\). Integrating from \(\varepsilon\) to \(r\), and then letting \(\varepsilon\to0^+\), gives
    \[
        2\int_0^r \frac1{tJ(t)} \int_{\partial B_t} \left[ t\partial_\nu u-D(t)u \right]^2 \,\dH{n-1}\,\dx{t} \le H_\beta(r)-H(0^+) <+\infty.
    \]
    This proves \eqref{eq:radial_defect_integrability}.
\end{proof}

With the help of the existence of the limit \(H(0^+)\), we can prove 

\begin{proposition}\label{prop:dyadic_J_V_estimates}
    Fix \(\beta\in(0,1)\), there exists \(r_\beta>0\) sufficiently small such that the following statements hold. 
    \begin{enumerate}[label=(\roman*)] 
        \item \label{item.integrable} For every \(r\in(0,r_\beta)\), 
        \begin{equation}\label{eq:V_square_integrability} 
            \int_0^r \frac{V^2(t)}{t}\,\dx{t} <+\infty. 
        \end{equation} 
        \item \label{item.Jcomparability} The function 
        \[ 
            r\mapsto e^{\frac2\beta r^\beta}(r^{-n-2}J(r)) 
        \] 
        is nondecreasing on \((0,r_\beta)\). Moreover, there exists a constant \(C_1>1\) such that 
        \begin{equation}\label{eq:J_dyadic_comparability} 
            C_1^{-1}s^{-n-2}J(s) \le r^{-n-2}J(r) \le C_1s^{-n-2}J(s)
        \end{equation} 
        for every \(r\in(0,r_\beta)\) and every \(s\in[r/2,r]\). In particular, 
        \[ 
            J(2r)\le C_12^{n+2}J(r) \qquad\text{ for }0<r<r_\beta/2. 
        \] 
        \item \label{item.Vcomparability} There exists a constant \(C_2>0\) such that 
        \begin{equation}\label{eq:V_dyadic_comparability} 
            \sup_{s\in[r/2,r]}V(s) \le C_2V(r) 
        \end{equation} for every \(r\in(0,r_\beta)\). 
    \end{enumerate} 
\end{proposition}

\begin{proof}
    We first prove \ref{item.integrable}. It follows from \eqref{eq:weak_frequency_formula} that \(H_\beta'(r) \ge \frac2r \left( r^\beta-V(r) \right)^2\) for a.e. \(r\in(0,r_\beta)\). Integrating from \(\varepsilon\) to \(r\), and then letting \(\varepsilon\to0^+\), we obtain
    \[ 
        \int_0^r \frac{ \left( t^\beta-V(t) \right)^2 }{t} \,\dx{t} <+\infty, 
    \] 
    because \(H(0^+)\) is finite. Since \(V(t)^2 \le 2\left(V(t)-t^\beta\right)^2 + 2t^{2\beta}\) and \(\beta\in(0,1)\), we get 
    \[ 
        \int_0^r \frac{V^2(t)}{t}\,\dx{t} \le 2\int_0^r \frac{ \left(V(t)-t^\beta\right)^2 }{t} \,\dx{t} + 2\int_0^r t^{2\beta-1}\,\dx{t} <+\infty,    
    \] 
    which proves \eqref{eq:V_square_integrability}.

    We next prove \ref{item.Jcomparability}. Set \(\mathcal{J}(r):=r^{-n-2}J(r)\). It then follows from \eqref{eq:J_log_derivative} and \eqref{eq:D_H_V_identity} that 
    \begin{equation}\label{eq.derivativecJrlog}
        \frac{\dx{}}{\dx r}\log\mathcal J(r) = \frac2r \left( H_\beta(r)+V(r)-r^\beta \right)
    \end{equation}
    for a.e. \(r\in(0,r_\beta)\). Since \(H_\beta(r)\ge0\) and \(V(r)\ge0\), we infer from \eqref{eq.derivativecJrlog} that \((\log\cJ(r))'\ge-2r^{\beta-1}\). This implies that
    \[ 
        \frac{\dx{}}{\dx r} \left[ \log\mathcal J(r)+\frac2\beta r^\beta \right] \ge0. 
    \] 
    Equivalently, \(r\mapsto e^{\frac2\beta r^\beta}\mathcal J(r)\) is nondecreasing. We now prove the two-sided estimate \eqref{eq:J_dyadic_comparability}. Since \(H_\beta(r)\to H(0^+)\), after decreasing \(r_\beta\), we may assume \(H_\beta(r)\le H(0^+)+1=:M_\beta\) for \(0<r<r_\beta\). Let \(s\in[r/2,r]\). It follows from \((\log\cJ(t))'\ge -2t^{\beta-1}\) that
    \[
        \log\mathcal J(r)-\log\mathcal J(s) 
        \ge -2\int_s^r t^{\beta-1}\,\dx{t} 
        = -\frac2\beta\left(r^\beta-s^\beta\right) 
        \ge -\frac2\beta r_\beta^\beta .
    \] 
    Hence \(\cJ(r)\ge e^{-\frac2\beta r_\beta^\beta}\cJ(s)\). For the reverse inequality, we first control \(V\) on dyadic intervals.  Using \eqref{eq:weak_frequency_formula} once again, 
    \[
        2\int_{r/2}^r \frac{ \left(t^\beta-V(t)\right)^2 }{t} \,\dx{t} \le H_\beta(r)-H_\beta(r/2) \le M_\beta.
    \]
    Therefore
    \[
        \int_{r/2}^r \frac{V^2(t)}{t}\,\dx{t} 
        \le 2\int_{r/2}^r \frac{ \left(t^\beta-V(t)\right)^2 }{t} \,\dx{t} + 2\int_{r/2}^r t^{2\beta-1}\,\dx{t}
        \le M_\beta + \frac{1}{\beta}r_\beta^{2\beta} =:C_\beta .
    \]
    By Cauchy's inequality, 
    \[
        \int_{r/2}^r\frac{V(t)}{t}\,\dx{t} 
        \le \left( \int_{r/2}^r\frac{V^2(t)}{t}\,\dx{t} \right)^{1/2} \left( \int_{r/2}^r\frac{\dx{t}}{t} \right)^{1/2}
        \le \sqrt{C_\beta\log2}.
    \]
    Now integrate \eqref{eq.derivativecJrlog} from \(s\) to \(r\), and discard the negative term \(-2t^{\beta-1}\). Since \(s\in[r/2,r]\), we obtain 
    \[
        \log\mathcal J(r)-\log\mathcal J(s) \le 2\int_s^r\frac{H_\beta(t)}{t}\,\dx{t} + 2\int_s^r\frac{V(t)}{t}\,\dx{t} 
        \le 2M_\beta\log2 + 2\sqrt{C_\beta\log2}. 
    \]
   Combining the two estimates proves \eqref{eq:J_dyadic_comparability} for a constant \(C_1>1\). Taking \(r\) replaced by \(2r\) and \(s=r\) gives \(\cJ(2r)\le C_1\cJ(r)\) for \(0<r<r_\beta/2\).

   It remains to prove \ref{item.Vcomparability}. Let \( G(r) := \int_{B_r} x_n^+\bigl(1-\chi_{\{u>0\}}\bigr)\,\dx{x}\ge 0\). Then \(G\) is nondecreasing in \(r\) and \(V(r)=r^{-n-1}G(r)/\cJ(r)\) where \(\cJ(r):=r^{-n-2}J(r)\). Since \(G(s)\le G(r)\), we get
   \[
        V(s) = \frac{s^{-n-1} G(s)}{\cJ(s)} 
        \le 
        \frac{s^{-n-1} G(r)}{\cJ(s)} 
        = \left(\frac rs\right)^{n+1} \frac{\cJ(r)}{\cJ(s)} V(r). 
   \]
   Since \(s\in[r/2,r]\), we have \((r/s)^{n+1}\le2^{n+1}\). It follows from \eqref{eq:J_dyadic_comparability} that \(\frac{\cJ(r)}{\cJ(s)}\le C_1\). Therefore \(V(s) \le 2^{n+1}C_1V(r)\). Taking the supremum over \(s\in[r/2,r]\) proves \eqref{eq:V_dyadic_comparability}.
\end{proof}

We immediately obtain 

\begin{corollary}\label{cor:V_vanishes} 
    \( \lim_{r\to0^+}V(r)=0 \).
\end{corollary}

\begin{proof} 
    By Corollary \ref{cor:weak_frequency_monotonicity}, \(r\mapsto H_\beta(r)\) is nondecreasing and has a finite right limit at \(0\). Thus \(d(r):=H_\beta(2r)-H_\beta(r)\ge 0\) and \(d(0^+)=0\). It follows from \eqref{eq:weak_frequency_formula} that \(2\int_r^{2r}\frac{(t^\beta-V(t))^2}{t}\le d(r)\). Thus, we have 
    \[
        \int_r^{2r}\frac{V^2(t)}{t}\,\dx{t} 
        \le d(r) + 2\int_r^{2r}t^{2\beta-1}\,\dx{t} 
        = d(r) + \frac{2^{2\beta}-1}{\beta}r^{2\beta} =:\varepsilon_\beta(r).
    \]
    Obviously \(\varepsilon_\beta(r)\to0\) as \(r\to 0^+\). Hence there exists \(s_r\in[r,2r]\) such that \(V^2(s_r) \le \frac{\varepsilon_\beta(r)}{\log2}\). Indeed, otherwise
    \[
        \int_r^{2r}\frac{V^2(t)}{t}\,\dx{t} > \frac{\varepsilon_\beta(r)}{\log2} \int_r^{2r}\frac{\dx{t}}t = \varepsilon_\beta(r),
    \]
    which is impossible. Since \(s_r\in[r,2r]\), we have \(r\in[s_r/2,s_r]\). Applying \eqref{eq:V_dyadic_comparability} with radius \(s_r\), we obtain 
    \[ 
        V(r) \le \sup_{t\in[s_r/2,s_r]}V(t) \le C V(s_r) \le C \left( \frac{\varepsilon_\beta(r)}{\log2} \right)^{1/2}. 
    \] 
    Letting \(r\to0^+\), we conclude that \(V(r)\to 0\). 
\end{proof}

We then obtain the existence of \(D(0^+)\) under the Assumption \ref{assumption:nonlinearity} on the nonlinearity \(f\).

\begin{proposition}[Existence of the frequency limit] \label{prop:frequency_limit} 
    Let \(u\) be a variational solution of \eqref{eq.wwp}, assume that \(f\) satisfies Assumption \ref{assumption:nonlinearity}, and let \(x^0\in\Sigma^u\). Then the finite limit 
    \[ 
        D(0^+) := \lim_{r\to0^+}D(r) 
    \] 
    exists. Moreover, \(D(0^+)=H(0^+)+\frac32\) and in particular \(D(0^+)\ge 3/2\).
\end{proposition}

\begin{proof}
    It follows from \eqref{eq:D_H_V_identity} that \(D(r)=H_\beta(r)+V(r)+\frac 32 -r^\beta\). Passing to the limit as \(r\to0^+\), we deduce from Corollary \ref{cor:weak_frequency_monotonicity} that \(H_\beta(r)\to H(0^+)\). Corollary \ref{cor:V_vanishes} gives \(V(r)\to 0\) and clearly \(r^\beta\to0\) since \(\beta>0\). Therefore 
    \[ 
        \lim_{r\to0^+}D(r) = H(0^+)+\frac32. 
    \] 
    Since \(H_\beta(r)\ge0\), we have \(H(0^+)\ge0\). Hence \(D(0^+)\ge\frac32\).
\end{proof}

We can also prove the first part of Theorem \ref{thm:main}, that \(\{\phi_r\}_{r>0}\) defined in \eqref{eq.phir} is bounded in \(W^{1,2}(B_1)\). 

\begin{corollary}\label{Corollary:frequency_bounded_nd}
    Let \(x^0\in\Sigma^u\), and let \(r_k\to 0^+\). Define 
    \[
        A_k
        :=
        \left(r_k^{1-n}J(r_k)\right)^{1/2},
        \qquad
        \phi_k(x)
        :=
        \frac{u(x^0+r_kx)}{A_k}.
    \]
    Then, for \(k\) sufficiently large, there is a constant \(C_n>0\) depending only on \(n\) such that
    \[
        \|\phi_k\|_{W^{1,2}(B_1)}^2
        \le
        C_n\left(1+D(0^+)\right).
    \]
    In particular, since \(D(0^+)\ge 3/2\),
    \[
        \|\phi_k\|_{W^{1,2}(B_1)}^2
        \le
        C_n D(0^+).
    \]
    Consequently, \(\{\phi_k\}\) is bounded in \(W^{1,2}(B_1)\).
\end{corollary}

\begin{proof}
    For simplicity, write \(r=r_k\), \(A=A_k\), and \(\phi(x):=A^{-1}u(x^0+rx)\). By the definition of \(A\), we have \(\int_{\partial B_1}\phi^2=A^{-2}r^{1-n}J(r)=1\). It follows from the identity \eqref{eq:M_minus_M0_identity}, and from the definitions of \(H_\beta\) in \eqref{eq.Hbetar} and \(V(r)\) in \eqref{eq.Vr} that
    \[
        \frac r{J(r)} \int_{B_r(x^0)}\left( |\nabla u|^2-uf(u) \right)\,\dx{x}
        =
        H_\beta(r)+V(r)-r^\beta+\frac32.
    \]
    Since \(H_\beta(r)+V(r)-r^\beta+\frac32\to D(0^+)\) as \(r\to0^+\), for all sufficiently small \(r\) we have
    \[
        \frac r{J(r)}\, \int_{B_r(x^0)}\left( |\nabla u|^2-uf(u) \right)\,\dx{x}
        \le
        D(0^+)+1.
    \]
    A direct computation gives that 
    \[
        \begin{aligned}
            \int_{B_1}|\nabla\phi|^2\,\dx{x}
            &=
            \frac{r^{2-n}}{A^2}
            \int_{B_r(x^0)}|\nabla u|^2\,\dx{x}\\
            &=
            \frac{r}{J(r)}
            \int_{B_r(x^0)}|\nabla u|^2\,\dx{x}\\
            &=
            \frac{r}{J(r)}\int_{B_r(x^0)}\left( |\nabla u|^2-uf(u) \right)\,\dx{x}
            +
            \frac{r}{J(r)}
            \int_{B_r(x^0)}u f(u)\,\dx{x} .
        \end{aligned}
    \]
    Since \(u\ge 0\), we have \(u f(u)\le u f^+(u)\). Moreover, it follows from Assumption \ref{assumption:nonlinearity} that \(f^+(u(x))\le \mathsf P_0\, u(x)\) in \(B_r(x^0)\). Hence,
    \[
        \frac{r}{J(r)}
        \int_{B_r(x^0)}u f(u)\,\dx{x}
        \le
        \frac{r\mathsf P_0}{J(r)}\int_{B_r(x^0)}u^2\,\dx{x}
        = r^2\mathsf P_0\,\int_{B_1}\phi^2\,\dx{x}.
    \]
    Using the trace inequality on \(B_1\), we have 
    \[
        \int_{B_1}\phi^2\,\dx{x}
        \le
        C_n\left(
        \int_{B_1}|\nabla\phi|^2\,\dx{x}
        +
        \int_{\partial B_1}\phi^2\,\dH{n-1}
        \right)
        =
        C_n\left(
        \int_{B_1}|\nabla\phi|^2\,\dx{x}
        +
        1
        \right).
    \]
    Since \(\mathsf P_0<\infty\), \(r^2\mathsf P_0\to0\) as \(r\to0^+\). We may shrink \(r\) if necessary so that \(C_n r^2\mathsf P_0\le 1/2\). Combining the above estimates gives
    \[
        \int_{B_1}|\nabla\phi|^2\,\dx{x}
        \le
        C_n\left( 
            1+D(0^+)
        \right),
    \]
    and
    \[
        \int_{B_1}\phi^2\,\dx{x}
        \le
        C_n\left(
        \int_{B_1}|\nabla\phi|^2\,\dx{x}
        +
        1
        \right)
        \le
        C_n\left( 
            1+D(0^+)
        \right).
    \]
    Thus, \(\|\phi\|_{W^{1,2}(B_1)}^2\le C_n(1+D(0^+))\). Since Proposition \ref{prop:frequency_limit} gives  \(D(0^+)\ge 3/2\), the constant \(1\) can be absorbed into \(D(0^+)\). Thus, \(\|\phi\|_{W^{1,2}(B_1)}^2\le C_n D(0^+)\). This concludes the proof.
\end{proof}

\section{Strong convergence in two dimensions}\label{sec:strong_convergence}

In this section we restrict to the two-dimensional case. The argument is based on the concentration-compactness theorem of Evans--M\"{u}ller \cite[Theorem 1.1 and 3.1]{EM94}. Compared with the proof in \cite{VW12}, the only additional point is the treatment of the nonlinear term under the frequency normalization. This is where the structural assumptions on \(f\) are used.

\begin{theorem}[Strong convergence of frequency-normalized blow-ups]\label{thm:strong-convergence}
    Assume \(n=2\), let \(u\) be a variational solution of \eqref{eq.wwp}, let \(x^0\in\Sigma^u\), and assume that \(f=f(z)\) satisfies Assumption \ref{assumption:nonlinearity}. Let \(r_k\to0^+\), and define
    \[ 
        A_k := \sqrt{r_k^{-1}J(r_k)}, \qquad \phi_k(x) := \frac{u(x^0+r_kx)}{A_k}. 
    \]  
    Then, after passing to a subsequence, \(\phi_k\to\phi_0\) strongly in \(W^{1,2}_{\mathrm{loc}}(B_1\setminus\{0\})\). Moreover, \(\phi_0\Delta\phi_0=0\) as a Radon measure in \(B_1\setminus\{0\}\).
\end{theorem}

\begin{proof}
    It follows from Corollary \ref{Corollary:frequency_bounded_nd} that the sequence \(\{\phi_k\}\) is bounded in \(W^{1,2}(B_1)\). Then \(\phi_k\rightharpoonup\phi_0\) weakly in \(W^{1,2}(B_1)\) and \(\phi_k\to\phi_0\) strongly in \(L^2(B_1)\). Let us denote \(D_0:=D(0^+)\). It follows from the weak formulation of \eqref{eq.wwp} that \(\Delta u + f(u)\chi_{\{u>0\}}=\mu\), where \(\mu\) is a nonnegative Radon measure supported on \(\partial\{u>0\}\). Therefore
    \[
        \Delta\phi_k = -\frac{r_k^2}{A_k} f(A_k\phi_k)\chi_{\{\phi_k>0\}} + \mu_k,
    \]
    where \(\mu_k\ge0\) is the rescaled free-boundary measure. Fix \(0<R<1\). Since \(A_k\phi_k(x)=u(x^0+r_kx)\), we infer from the growth condition \eqref{eq.growth_assumption} that \(A_k\phi_k(x)\le C(r_kR)^{3/2}\) for \(x\in B_R\) and \(k\) sufficiently large. Hence the structural assumption on the positive part implies
    \[
        f^+(A_k\phi_k)\le \mathsf P_0\, A_k\phi_k.
    \]
    Consequently, 
    \[
        \Delta\phi_k \ge -r_k^2\mathsf P_0\,\phi_k+\mu_k \qquad\text{ in }B_R,
    \]
    where \(\mathsf P_0\) is the nonnegative constant from \eqref{eq.P0N0} in Assumption \ref{assumption:nonlinearity}. Set \(g_k:=-r_k^2\mathsf P_0\,\phi_k\). Since \(\{\phi_k\}\) is bounded in \(L^2(B_R)\), we obtain that \(g_k\to 0\) strongly in \(L^2(B_R)\). Thus \(\nu_k:=\Delta\phi_k-g_k\) is a nonnegative Radon measure in \(B_R\). Let \(0<\sigma<\tau<1\), and let \(p_k\) be the solution of the Poisson equation 
    \[
        \begin{cases}
            \Delta p_k=g_k &\text{ in }B_\tau,\\[3pt]
            p_k=0 &\text{ on }\partial B_\tau.
        \end{cases}
    \]
    It follows from the \(L^2\) Calder\'{o}n--Zygmund estimate that \(p_k\to0\) strongly in \(W^{1,2}(B_\tau)\). Hence \(w_k:=\phi_k-p_k\) satisfies \(\Delta w_k=\nu_k\ge0\) in \(B_\tau\), and \(w_k-\phi_k\to0\) strongly in \(W^{1,2}(B_\tau)\). After a standard mollification inside \(B_\sigma\), we may apply the Evans--M\"{u}ller concentration-compactness theorem \cite[Theorem 1.1 and Theorem 3.1]{EM94} to the subharmonic sequence \(w_k\). It follows that \(\partial_1\phi_k\,\partial_2\phi_k \to \partial_1\phi_0\,\partial_2\phi_0\) in the sense of distributions in \(B_\sigma\), and \((\partial_1\phi_k)^2-(\partial_2\phi_k)^2 \to (\partial_1\phi_0)^2-(\partial_2\phi_0)^2\) in the sense of distributions in \(B_\sigma\). It follows from \eqref{eq:radial_defect_integrability}, \(H_\beta(r)\to H(0^+)\) and the statement \ref{item.Vcomparability} of Proposition \ref{prop:dyadic_J_V_estimates} that for every \(0<\rho<\sigma<1\),
    \[
        \nabla\phi_k\cdot x-D_0\phi_k \to0 \qquad\text{ strongly in }L^2(B_\sigma\setminus B_\rho).
    \]
    Since \(\phi_k\to\phi_0\) strongly in \(L^2(B_1)\), we obtain \(\nabla\phi_k\cdot x \to \nabla\phi_0\cdot x\) strongly in \(L^2(B_\sigma\setminus B_\rho)\). Note that if \(n=2\), the identities
    \[
        |\nabla\phi_k\cdot x|^2 = x_1^2(\partial_1\phi_k)^2 + x_2^2(\partial_2\phi_k)^2 + 2x_1x_2\partial_1\phi_k\partial_2\phi_k,
    \]
    and \((\partial_1\phi_k)^2-(\partial_2\phi_k)^2\) allow one to recover the distributional limits of \((\partial_1\phi_k)^2\) and \((\partial_2\phi_k)^2\) on every annulus \(B_\sigma\setminus B_\rho\). Thus \(|\nabla\phi_k|^2 \to |\nabla\phi_0|^2\) in the sense of distributions on \(B_\sigma\setminus B_\rho\). Since \(\nabla\phi_k\rightharpoonup\nabla\phi_0\) weakly in \(L^2\), convergence of the local \(L^2\)-norms implies \(\nabla\phi_k\to\nabla\phi_0\) strongly in \(L^2(B_\sigma\setminus B_\rho)\). Together with the strong \(L^2\)-convergence, this proves \(\phi_k\to\phi_0\) strongly in \(W^{1,2}(B_\sigma\setminus B_\rho)\). Since \(\rho\) and \(\sigma\) are arbitrary, we conclude that \(\phi_k\to\phi_0\) strongly in \(W^{1,2}_{\mathrm{loc}}(B_1\setminus\{0\})\).

    The proof that \(\phi_0\Delta\phi_0=0\) as a Radon measure in \(B_1\setminus\{0\}\) is also different from the one in \cite{VW12} due to our new structural assumptions on \(f\). We provide the details here. Let \(\eta\in C_c^\infty(B_1\setminus\{0\})\), \(\eta\ge0\). By the strong convergence just proved, \(\langle\Delta\phi_k,\eta\phi_k\rangle \to \langle\Delta\phi_0,\eta\phi_0\rangle\). Since \(\mu_k\) is supported on \(\partial\{\phi_k>0\}\), and \(\phi_k=0\) on this support, we have \(\langle\mu_k,\eta\phi_k\rangle=0\). Therefore
    \[
        \left| \langle\Delta\phi_k,\eta\phi_k\rangle \right| \le r_k^2 \int_{\operatorname{spt}\eta} \eta\,\phi_k^2 \frac{|f(A_k\phi_k)|}{A_k\phi_k} \,\dx{x},
    \]
    where the quotient is interpreted as \(0\) on \(\{\phi_k=0\}\). The positive part is estimated by
    \[
        r_k^2 \int_{\operatorname{spt}\eta} \eta\,\phi_k^2 \frac{f^+(A_k\phi_k)}{A_k\phi_k} \,\dx{x} \le Cr_k^2\mathsf P_0 \int_{\operatorname{spt}\eta}\phi_k^2\,\dx{x} \to0.
    \]
    For the negative part, Assumption \ref{assumption:nonlinearity} gives \(f^-(z)\le \mathsf N_0\, z(1+|\log z|),\) for \(0<z<z_0\). Thus
    \[
        r_k^2 \int_{\operatorname{spt}\eta} \eta\,\phi_k^2 \frac{f^-(A_k\phi_k)}{A_k\phi_k} \,\dx{x}
        \le Cr_k^2\mathsf N_0 \int_{\operatorname{spt}\eta} \phi_k^2 \bigl(1+|\log A_k|+|\log\phi_k|\bigr) \,\dx{x}.
    \]
    A direct calculation shows that 
    \[
        \frac{\dx{}}{\dx{r}}\,\log A(r)=\frac{D(r)}r, \qquad A(r):=\left(r^{-1}J(r)\right)^{1/2}. 
    \]
    Since \(D(r)\to D_0<+\infty\), we have 
    \[
        |\log A_k| \le C\left(1+\log\frac1{r_k}\right). 
    \]
    Moreover, in dimension two, the uniform \(W^{1,2}\)-bound gives a uniform \(L^p\)-bound for every finite \(p\), and the elementary inequality \(t^2|\log t|\le C_p(1+t^p)\), \(t\ge 0\), implies 
    \[
        \int_{\operatorname{spt}\eta}\phi_k^2|\log\phi_k|\,\dx{x}\le C.
    \]
    Therefore 
    \[
        r_k^2 \int_{\operatorname{spt}\eta} \eta\,\phi_k^2 \frac{f^-(A_k\phi_k)}{A_k\phi_k} \,\dx{x} \le C r_k^2\left(1+\log\frac1{r_k}\right) \to0.
    \]
    Hence \(\langle\Delta\phi_0,\eta\phi_0\rangle=0\) for every nonnegative \(\eta\in C_c^\infty(B_1\setminus\{0\})\). This is equivalent to \(\phi_0\Delta\phi_0=0\) as a Radon measure in \(B_1\setminus\{0\}\). The proof is complete.
\end{proof}

By Theorem \ref{thm:strong-convergence}, the vanishing of \(V(r)\), and the strong convergence in \(W^{1,2}_{\mathrm{loc}}(B_1\setminus\{0\})\), we may pass to the domain-variation identity in the upper half-plane. The classification argument of \cite[Theorem 9.1]{VW12} therefore applies verbatim: \(D(0^+)\) is an integer \(N(x^0)\ge2\), and every frequency-normalized blow-up is equal to the profile stated in Theorem \ref{thm:main}. In particular, the blow-up is unique and the convergence holds for the full family \(r\to0^+\).

Moreover, since \(H_\beta\ge0\) and \(V\ge0\),
\[
    D_{y,u}(r)\ge\frac32-r^\beta
\]
at every \(y\in\Sigma^u\), for all sufficiently small \(r\). On compact subsets, the admissible radius may be chosen uniformly. This estimate replaces \cite[Theorem 6.12(i)]{VW12} in the proof of \cite[Theorem 9.2]{VW12}, and hence \(\Sigma^u\) is locally finite in two dimensions. Finally, under the additional hypotheses of Theorem \ref{thm:main2}, the argument of \cite[Theorem 10.1]{VW12} excludes \(\Sigma^u\). This completes the proofs of Theorems \ref{thm:main} and \ref{thm:main2}.

\subsubsection*{Acknowledgement}
\hyphenpenalty=10
\sloppy
This work is supported by National Natural Science Foundation of China under Grants 12125102, 12526202, Nature Science Foundation of Guangdong Province under Grant 2024A1515012794, and Shenzhen Science and Technology Program (JCYJ20241202124209011).

\subsubsection*{Data availability} No data were used in this research.


\end{document}